\documentclass[
    a4paper,
    10pt,
]{article}

\usepackage{custom-style}
\usepackage{math}
\usepackage[plain]{theoremtemplate-style}
\usepackage[margin=2cm]{geometry}
\setuptodonotes{fancyline, color=blue!30}

\usepackage[style=alphabetic,backend=biber,maxbibnames=6]{biblatex}
\title{$L_\infty$-morphisms between twisted Courant $r$-Lie algebras and untwisted Courant $(r{+}1)$-Lie algebroids} 
\author{Domenico Fiorenza, Antonio Michele Miti}

\begin{document}
%#########################################################################%
\maketitle

%=================================================#
\begin{abstract}
  In {\it Lie infinity algebras and higher analogues of Dirac structures and Courant algebroids'} [Journal of Symplectic Geometry, Vol. 10, no. 4 (2012), pp. 563-599], Marco Zambon constructs an $L_\infty$-algebra associated with any higher standard or twisted Courant algebroid (also known as a Vinogradov algebroid), and exhibits an explicit $L_\infty$-morphism from the Lie algebra associated with a standard Lie algebroid twisted by a closed 2-form to the Lie-2 algebra of the standard Courant algebroid. He poses the question of whether analogous canonical $L_\infty$-morphisms exist in higher degrees---namely, for any standard higher Courant algebroid twisted by a closed $(r+1)$-form. We adfirmatively answer this question, presenting 
a general framework that naturally yields such canonical $L_\infty$-morphisms for arbitrary $r$, while at the same time clarifying the geometrical and homotopical structures underlying the construction. We also show how this framework accommodates the canonical morphism between Roger's observable $L_\infty$-algebra of a pre-$r$-plectic manifold and the higher Courant algebra described by Zambon and one of the authors in {\it Observables on multisymplectic manifolds and higher Courant algebroids} [Communications in Contemporary Mathematics Vol. 27, No. 08, 2550006 (2025)].

\end{abstract}
%=================================================#

\tableofcontents

%=================================================#
\section{Introduction}
\label{sec:intro}
%=================================================#
%

In the late eighties, Courant~\cite{Courant1990}, Dorfman~\cite{Dorfman1987}, and Weinstein introduced the notion of \emph{Courant algebroid} as a unifying framework encompassing both symplectic and Poisson geometry through the language of Dirac structures.  
A decade later, Roytenberg and Weinstein~\cite{Roytenberg1998} showed that every Courant algebroid gives rise to a two-term $L_\infty$-algebra, thus revealing the homotopical nature underlying these structures.

In the early 2010s, building on \cite[\S 3.8]{Gualtieri2004}, Zambon~\cite{Zambon2012} proposed a higher analogue of Courant algebroids, together with generalized Dirac structures and their associated $L_\infty$-algebras.  
In analogy with the original motivation of Courant and Dorfman, these higher Dirac structures provide a natural framework for describing $k$-plectic structures (or multisymlectic in the sense of ~\cite{Cantrijn1999}) on smooth manifolds.  
Around the same period, \v{S}evera and Weinstein~\cite{Severa2001} introduced the notion of a \emph{twisted} Courant bracket, obtained by deforming the Courant bracket by a closed three-form.  
Zambon's construction also generalizes this idea, showing that any closed $(k+1)$-form can be used to twist a higher Courant algebroid together with its corresponding $L_\infty$-algebra.

In the same way as one associates a Poisson algebra to any symplectic manifold, Rogers~\cite{Rogers2012} introduced, for every closed non-degenerate $(k+1)$-form $\sigma$, an $L_\infty$-algebra describing the \emph{multisymplectic observables} associated to the $k$-plectic manifold $(M,\sigma)$ (see also~\cite{Callies2016} for the degenerate case).  
Since a closed form can be viewed as a Dirac structure within a suitable higher Courant algebroid, it is natural to compare the $L_\infty$-algebras arising from these two contexts.  
In the $2$-plectic case ($\sigma \in \Omega^3(M)$), Rogers~\cite{Rogers2013} showed that the $L_2$-algebra of observables embeds into the $L_2$-algebra associated to the $\sigma$-twisted Courant algebroid.  
On the other hand, Zambon~\cite[Thm. 7.1]{Zambon2012} constructed a natural morphism from the Lie algebra of a twisted (with respect to a closed $2$-form) Lie algebroid to the $L_2$-algebra associated with a standard untwisted Courant algebroid.
A higher analogue of the first embedding was conjectured by Sämann and Ritter~\cite{Ritter2015a} and recently constructed explicitly by Miti and Zambon~\cite{Miti2024}.  
The higher version of the second morphism was also conjectured in~\cite{Zambon2012}, but to our knowledge, no proof has been available so far.

The explicit construction in~\cite{Miti2024} provides a direct higher-dimensional generalization of Rogers' morphism, but a general conceptual framework explaining its origin has been missing.  
Furthermore, the higher analogue of Zambon's morphism from the twisted Courant algebra to the untwisted Lie algebroid has remained conjectural.  
In this paper, we address these open problems by introducing in this context the technology of \emph{homotopy limits} within the model category of differential graded Lie algebras (DGLA) and, equivalently, of $L_\infty$-algebras.  
Our goal is to provide a conceptual framework of the construction in~\cite{Miti2024} and, in doing so, to establish Zambon's conjectured morphism between the twisted and untwisted Courant algebras, thereby addressing the questions raised in~\cite[\S9]{Zambon2012}.

\medskip
To describe the setting, let $\sigma \in \Omega^{r+1}(M)$ be a closed $(r{+}1)$-form.  
We consider the following diagram of $L_\infty$-algebras:
\[
\begin{tikzcd}
    \Rogers[\sigma]{r-1} \ar[r] & \Courant[\sigma]{r-1} \ar[r] & \Courant[]{r}.
\end{tikzcd}
\]
Here $\Rogers[\sigma]{r-1}$ denotes Rogers’ $L_\infty$-algebra of multisymplectic observables~\cite{Rogers2012}, while $\Courant[\sigma]{r-1}$ and $\Courant[]{r}$ denote the higher and twisted Courant algebras introduced by Zambon~\cite{Zambon2012}.  
The first arrow coincides with the morphism of~\cite{Rogers2013} for $r=1$ and with its higher-dimensional generalization constructed in~\cite{Miti2024} for $r>1$;  
the second arrow corresponds to the morphism described in~\cite{Zambon2012} for $r=1$ and conjectured there for $r>1$.

The key observation of the present work is that all three objects in the above diagram can be realized as \emph{homotopy fibers} of DGLA or, more generally, $L_\infty$-algebra morphisms.  
This fact, already established for $\Rogers[\sigma]{r-1}$ in~\cite{Fiorenza2014}, extends to the Courant case via Zambon's construction, which relies on Getzler's model~\cite{Getzler2010}. 
The latter can be seen as a specialization of the cone construction of Fiorenza–Manetti~\cite{Fiorenza2007}, see also~\cite{Bandiera2015}.  
    This observation allows us to bypass the explicit combinatorial complexity of the higher multibrackets in the higher Courant $L_\infty$-algebras: their structure can instead be deduced from the inclusion morphism $\gLarge_{r,\,\sigma}^{\geq 0}\hookrightarrow \gLarge_{r,\,\sigma}$ of the nonnegative degree part $\gLarge_{r,\,\sigma}^{\geq 0}$ of a suitable DGLA $\gLarge_{r,\,\sigma}$. The $L_\infty$-algebra morphism leading to Rogers' algebra was already identified in~\cite{Fiorenza2014}.  
    
We simplify the DGLA embedding appearing in Zambon's construction by introducing a \emph{Cartan DGLA}, which encodes the Cartan calculus on the underlying manifold and plays a central role in our homotopical interpretation.
Having shifted our attention to morphisms of DGLA's, it becomes natural to ask whether the corresponding morphisms between $L_\infty$-algebras can be obtained from (homotopy) commutative squares of DGLA's or, more generally, of $L_\infty$-algebras.  
%Using the equivalence between the homotopy category of DGLAs and that of $L_\infty$-algebras, 
We show that this is indeed the case.
%The morphism introduced by Rogers in the 2-plectic case~\cite{Rogers2013}, conjectured in~\cite{Ritter2015a}, and constructed in full generality in~\cite{Miti2024} is a morphism between homotopy fibers, naturally induced by such a commutative square.

\medskip
Our main results can be summarized as follows.
\\
Denoting by $\Cartan$ the sub-DGLA of derivations of the de Rham algebra $\Omega(M)$ of $M$ generated by contractions with vector fields, with differential given by $[\ddR,-]$, we consider the semidirect product  $\gSmall_{r} \coloneqq \Cartan \ltimes \Omega[r]$. The element $\sigma\in \Omega^{r+1}_{\mathrm{cl}}(M)$ is a Maurer-Cartan element in $\gSmall_{r}$, hence one can define the $\gSmall_{r,\sigma}$ as the $\sigma$-twisted DGLA obtained from $\gSmall_{r}$.
This DGLA $\gSmall_{r,\sigma}$ is seen to be a sub-DGLA of the DGLA $\gLarge_{r,\sigma}$ appearing in Zambon's construction and we prove:
% In \autoref{def:CartanAlgebra} we introduce the $2$-term DGLA $\Cartan$ consisting of all possible contraction operators $\iota_X$ and Lie derivative operators $\Lie_X$ for any vector field $X$ on $M$.
% The action of $\Cartan$ on the dg-space $\Omega(M)$ encodes the Cartan calculus on the manifold $M$.
% Let $r\geq 1$ be a fixed integer and consider the DGLA $\gSmall_{r} \coloneqq \Cartan \ltimes \Omega[r]$.
% Any closed $r{+}1$-form $\sigma \in \Omega^{r+1}_{\mathrm{cl}}(M)$ defines a Maurer-Cartan element in $\gSmall_{r}$, hence one can define the $\gSmall_{r,\sigma}$ as the $\sigma$-twisted DGLA. 
% The higher Courant algebra $\Courant[\sigma]{r{-}1}$ is the homotopy fiber associated  with the inclusion of $\gSmall_{r,\sigma}^{\ge 0}$ into $\gSmall_{r,\sigma}$.

    \begin{bigthm}[Smaller model for \texorpdfstring{$\Courant[\sigma]{r}$}\,\! ]
        \label{bigthm:smaller-model}
        The higher Courant algebra $\Courant[\sigma]{r{-}1}$ is a model for the homotopy fiber associated with the inclusion of DGLAs 
        \begin{displaymath}
            \gSmall_{r,\sigma}^{\ge 0} \hookrightarrow \gSmall_{r,\sigma}~,
        \end{displaymath}
        where $\gSmall_{r,\sigma}^{\ge 0}$ is the truncation $\gSmall_{r,\sigma}$ in non-negative degrees.
    \end{bigthm}
    In \autoref{lem:Linfty-morphism-smallgr-to-smallgr+1} we exhibit a $2$-term $L_\infty$-morphims $\Phi$, i.e., an $ L_\infty$-morphisms $\Phi$ comprising only a linear and quadratic component, between the DGLA $\gSmall_{r,\sigma}$ and the DGLA $\gSmall_{r{+}1}$.
    Such a morphism can be $\sigma$-twisted to a $L_\infty$-morphism $\Phi_\sigma$ between $\gSmall_{r,\sigma}$ and $\gSmall_{r{+}1}$ which is compatible with  truncations to non-negative degrees.
    This leads to the existence of the natural $L_\infty$-morphism %from the twisted to the untwisted Courant algebra for any $r\geq 1$, as 
    conjectured in~\cite{Zambon2012}.
    %, and it is realized as the universal morphism between homotopy fibers.
%
\begin{bigthm}[Morphism from twisted to untwisted Courant $L_\infty$-algebras]
    \label{bigthm:morphism-twistedCourant-untwistedCourant}
    There exists a distinguished $L_\infty$-morphism between the twisted and untwisted higher Courant $L_\infty$-algebras
    \begin{displaymath}
        \Courant[\sigma]{\,r-1}\ \longrightarrow\ \Courant{r}~,
    \end{displaymath}
    induced by the commutative square  of $L_\infty$-morphisms
    \begin{displaymath}
        \begin{tikzcd}[column sep =large]
            \gSmall_{r,\sigma}^{\ge 0} \ar[r,hook]\ar[d,"\Phi_\sigma|_{\ge 0}"'] & \gSmall_{r,\sigma} \ar[d,"\Phi_\sigma"] 
            \\
            \gSmall_{r+1}^{\ge 0} \ar[r,hook] & \gSmall_{r+1}.
        \end{tikzcd}
    \end{displaymath}
    via the universal property of homotopy fibers.
\end{bigthm}
    Finally, we notice that the $L_\infty$-algebra of multisymplectic observables $\Rogers[\sigma]{r-1}$ has been realized in \cite{Fiorenza2007} as the homotopy fiber associated with the multicontraction morphism $\iota_\infty \sigma$ of Hamiltonian vector fields $\X_{\ham,\sigma}$ with the given closed form $\sigma$.
    This leads to the $L_\infty$-embedding constructed in~\cite{Miti2024}.
    %Therefore, we prove that the $L_\infty$-embedding of the Rogers algebra into the twisted Courant $L_\infty$-algebra constructed in~\cite{Miti2024} arises functorially from a homotopy-commutative square of DGLAs, and hence admits a conceptual explanation in terms of the universal property of homotopy fibers.
%
\begin{bigthm}[Morphism from multisymplectic observables to twisted Courant $L_\infty$-algebras]
    \label{bigthm:morphism-Rogers-Courant}
    There exists a distinguished $L_\infty$-morphism from the $L_\infty$-algebra of multisymplectic observables to the twisted higher Courant $L_\infty$-algebra
    \begin{displaymath}
        \Rogers[\sigma]{r-1}\ \longrightarrow\ \Courant[\sigma]{r-1}~,
    \end{displaymath}
    induced by the homotopically commutative square of DGLAs inside the category of $L_\infty$-algebras 
    \begin{displaymath}
       \begin{tikzcd}[column sep =large]
             \X_{\ham,\sigma} \ar[r,"\iota_{\infty}\sigma"] \ar[d,"\Lie"'] &
            (\Omega^0\to\cdots\to \Omega^{r-1}\to \d\Omega^{r-1}) \ar[d,hook]
            \\
            \g_{r,\sigma}^{\ge 0} \ar[r,hook] \ar[ur,Rightarrow] &
            \g_{r,\sigma}~.
        \end{tikzcd}
    \end{displaymath}
    
\end{bigthm}

\medskip
The paper is organized as follows.  
Section~2 recalls the background on the $L_\infty$-algebras associated to multisymplectic forms, Courant algebroids, and their twisted versions.  
Section~3 describes the construction of a natural morphism between twisted and untwisted Courant algebras.  In particular, \autoref{bigthm:smaller-model} is proven as Proposition \ref{prop:smallg-and-largeg-same-fiber}, and \autoref{bigthm:morphism-twistedCourant-untwistedCourant} is proven as the concluding statement of subsection \ref{ssec:theoremB}.
Section~4 develops the homotopical framework that justifies the construction of the Rogers-to-Courant $L_\infty$-morphism.  
Section~5 contains conclusions and outlooks.

%+------------------------------------------------+
\subsection*{Notations and conventions}
%+------------------------------------------------+

We will tacitly assume all graded vector spaces to be $\mathbb{Z}$-graded, that is, that they are objects in the category $\Vect^{\mathbb{Z}}$. Given two graded vector spaces $V$ and $W$, we denote by $\underline{\Hom}(V, W)$ the graded vector space of homogeneous linear maps from $V$ to $W$, is defined by
\begin{equation*}
    \underline{\Hom}(V, W)^k := \Hom_{\Vect^{\mathbb{Z}}}(V, W[k])~,
\end{equation*}
where $W[k]$ is the $k$-th suspension of $W$, defined by $W[k]^n := W^{n+k}$ for all $n \in \mathbb{Z}$.
\\
In the case $W = V$, we write $\underline{\Hom}(V) := \underline{\Hom}(V, V)$ and refer to it as the graded associative algebra of endomorphisms of $V$, with composition of linear maps as the associative product denoted by $\circ$.

If $V$ and $W$ are graded associative algebras, we denote by $\Der(V, W) \subset \underline{\Hom}(V, W)$ the graded Lie algebra of derivations from $V$ to $W$, that is, graded linear maps $D: V \to W$ satisfying
\begin{equation*}
    D(a \cdot b) = D(a) \cdot b + (-1)^{|a|} a \cdot D(b)~,
\end{equation*}
for all homogeneous elements $a, b \in V$, and equipped with the commutator of linear operators as Lie bracket.

In particular, if $V$ is a differential graded associative algebra with differential $\d$, we denote by $\Der(V) := \Der(V, V)$ the differential graded Lie algebra of derivations of $V$, endowed with the differential $\delta := [\d, \blank]$.\\
We adopt the cohomological convention: all differential graded vector spaces are assumed to be cochain complexes, i.e., their differentials $\d$ have degree $+1$.

We will tacitly regard every differential graded Lie algebra (DGLA) as an $L_\infty$-algebra via the canonical inclusion of categories
$$
\Dgla \hookrightarrow \Linfty~.
$$
Recall that $\Dgla$ is not a full subcategory of $\Linfty$; there are more general $L_\infty$-morphisms between DGLAs than DGLA morphisms.
We say that a square of $L_\infty$-algebras and $L_\infty$-morphisms is \emph{homotopy commutative} if it is equipped with an explicit non-trivial $L_\infty$-homotopy filling the 2-cell, that is, it is of the form
\begin{displaymath}
\begin{tikzcd}
A \arrow[r, "f"] \arrow[d, "g"'] & B \arrow[d, "h"] \\
C \arrow[r, "k"'] \arrow[ur, Rightarrow, shorten <=4pt, shorten >=4pt, "\eta"] & D
\end{tikzcd}
\end{displaymath}
where $\eta$ is an $L_\infty$-homotopy between $h \circ f$ and $k \circ g$.

When the above homotopy $\eta$ is the identity, we simply say that the square is \emph{commutative}, and depict it as
\begin{displaymath}
\begin{tikzcd}
A \arrow[r, "f"] \arrow[d, "g"'] & B \arrow[d, "h"] \\
C \arrow[r, "k"'] & D
\end{tikzcd}
\end{displaymath}

Given a DGLA $(\g, \d,[\blank,\blank])$, we denote by $\MC(\g)$ the set of Maurer-Cartan elements of $\g$, i.e., degree $1$ elements $\alpha \in \g^1$ satisfying the Maurer-Cartan equation
\begin{equation*}
    \d \alpha + \frac{1}{2}[\alpha, \alpha] = 0~.
\end{equation*}
Any Maurer-Cartan element $\alpha \in \MC(\g)$ induces a \emph{"twisted"} differential $\d_\alpha= \d + \ad_{\alpha}$ on $\g$, where $\ad_{\alpha}:\g\to\g$ is the degree $1$ derivation defined by $\ad_{\alpha}(\eta) := [\alpha, \eta]$ for all $\eta \in \g$.
We say that two Maurer-Cartan elements $\alpha_0, \alpha_1 \in \MC(\g)$ are gauge equivalent if there exists a degree $0$ element $\xi \in \g^0$ such that
\begin{equation*}
    e^{\ad_{\xi}}\circ \big(\d_{\alpha_0}\big) = 
    \big( \d_{\alpha_1} \big) \circ e^{\ad_{\xi}}~,
\end{equation*}
where the exponential $e^{\ad_{\xi}}$ is defined via the usual power series, whenever it converges, e.g., when $\g$ is nilpotent.

As it is customary in infinitesimal deformation theory over a characteristic zero field $\k$, see e.g \cite{Manetti2022}, all DGLA's considered in the following are tacitly assumed to be nilpotent\footnote{i.e., there exists a natural number $k$ such that the $k$-th power of the adjoint action $\ad_\xi^k$ is identically zero for any element $\xi$} by considering a DGLA $\g$ as a placeholder for the collection of nilpotent DGLA's $\{ \g \otimes \mathfrak{m}_A\}_{A \in \Artin}$ where $\mathfrak{m} _A$ is the maximal ideal of a local Artin $\k$-algebra $A$.
%The exponential $\exp(\g)$ of a nilpotent Lie algebra $\g$ can be viewed as a group with the same underlying set as $\g$, equipped with the Baker–Campbell–Hausdorff product
%$$     x \bullet y = x + y + \tfrac{1}{2}[x,y] + \cdots~, $$
%since the nilpotency of $\g$ ensures that the series truncates to a finite sum.
The nilpotency of $\g$ ensures that the \emph{gauge action} of $a\in\g^0$ on $x\in\g^1$ as
\begin{equation}\label{eq:MC-Gauge-Transformation}
    e^a \star x = x + \sum_{n \ge 0} \frac{(\mathrm{ad}_a)^n}{(n+1)!}\big([a,x] - \d a\big)~
\end{equation}
is well defined.
This action is more generally defined whenever the series on the right-hand sides converges in $\g$ (or in a suitable completion of $\g$) under some notion of convergence.
The set of Maurer–Cartan elements is stable under the gauge action.
%, meaning that two Maurer-Cartan elements $x,y\in\MC(\g)$ are gauge equivalent if and only if there exists $a\in\g_0$ such that $y=e^a\star x$.
%
Remarkably, gauge equivalence of Maurer-Cartan elements coincides with their \emph{homotopy equivalence} between Maurer–Cartan elements, see \autoref{appendix:MC-homotopy} or \cite[\S 6.3]{Manetti2022} for details. 
We will make use of this fact in Section \autoref{sec:RogerstoCourant}.

For $M$ a smooth manifold, we denote by $\Omega(M)$ the graded-commutative associative algebra of differential forms on $M$, endowed with the wedge product $\wedge$ and with the de Rham differential $\ddR$, and by $\X(M)$ the Lie algebra of vector fields on $M$. When no ambiguity arises, we omit the explicit reference to $M$ and write $\Omega$ for $\Omega(M)$ and $\X$ for $\X(M)$. 
For later reference, let us recall that the \emph{Cartan calculus} on $M$ consists of the following distinguished graded derivations of $\Omega$:
\begin{itemize}
    \item the Lie derivative $\Lie_X$ along a vector field $X$, of degree $0$, i.e.\ $\Lie_X \in \Der(\Omega)^0$;
    \item the contraction (or interior product) $\iota_X$ along a vector field $X$, of degree $-1$, i.e.\ $\iota_X \in \Der(\Omega)^{-1}$;
    \item the de Rham differential $\ddR$, of degree $1$, i.e.\ $\ddR \in \Der(\Omega)^1$.
\end{itemize}
These operators satisfy the following identities for all vector fields $X, Y \in \mathfrak{X}$:
\begin{align}
    \ddR^2 &= 0~,\label{eq:CartanDsquared} \\
    \Lie_X \circ \ddR - \ddR \circ \Lie_X &= 0~,\label{eq:CartanLievsD} \\
    \ddR \circ \iota_X + \iota_X \circ \ddR &= \Lie_X~,\label{eq:CartanDvsIota} \\
    \Lie_X \circ \Lie_Y - \Lie_Y \circ \Lie_X &= \Lie_{[X, Y]_\X}~,\label{eq:CartanLiesquared} \\
    \Lie_X\circ \iota_Y - \iota_Y \circ\Lie_X &= \iota_{[X, Y]_\X}~,\label{eq:CartanLievsIota} \\
    \iota_X \circ\iota_Y + \iota_Y \circ \iota_X &= 0~,\label{eq:CartanIotasquared}
\end{align}
where $[\blank,\blank]_{\X}$ denotes the Lie bracket of vector fields.
We will often consider DGLAs $\g$ built from shifts of $\Omega$ and $\X$, equipped with the Cartan calculus operators above. Throughout, we denote their differential simply by $\d$ and their bracket by $[\blank,\blank]$, with the precise interpretation, namely $\ddR$, $[\blank,\blank]_{\X}$, or suitable combinations thereof, understood from the context.
For $\alpha\in\Omega[k]$, we set $|\alpha|=\deg(\alpha)-k$, where $\deg(\alpha)$ denotes the usual differential form degree on $M$.

%+------------------------------------------------+
\subsection*{Acwnowledgements}
%+------------------------------------------------+
This research is part of the activities of the Department of Excellence (Dipartimento di Eccellenza) Project CUP \href{https://sites.google.com/uniroma1.it/excellence-department/home}{B83C23001390001}.
Both authors are members of the {\it Gruppo Nazionale per le Strutture Algebriche, Geometriche e le loro Applicazioni} (GNSAGA–INdAM). 
AMM acknowledges funding from the {\it European Union's  Horizon 2020} research and innovation programme under Grant Agreement {\it No.~101034324}. 
The authors thank the Mittag-Leffler Institute for its hospitality during the 2025 program {\it Cohomological Aspects of Quantum Field Theory}, and Ruggero Bandiera, Ezra Getzler, Noriaki Ikeda, and Marco Zambon for helpful preliminary discussions and comments on a preliminary draft of this manuscript.

%=================================================#
\section{Background}
%\section{Higher Courant algebroids and \texorpdfstring{$L_\infty$}{L-infinity}-algebras}
%=================================================#

In this section, we recall the definitions of higher Courant algebroids, higher Dirac structures, twisted Courant algebroids, and multisymplectic manifolds. 
We review Zambon's construction \cite{Zambon2012} of the $L_\infty$-algebra associated to a higher Courant structure on a manifold $M$, obtained via a derived bracket construction on the DGLA of smooth functions on $T^*[r]T[1]M$, and Roger's construction \cite{Rogers2012} of the so-called $L_\infty$-algebra of observables.
\\
We assume the reader to be familiar with the basic notions of graded geometry, %in particular with the construction of the shifted cotangent bundle $T^*[r]T[1]M$ of a smooth manifold $M$ and its algebra of smooth functions $\mathcal{C}(T^*[r]T[1]M)$
$L_\infty$-algebras, homotopy fibers in the category of DGLAs, and $L_\infty$-algebras (in \autoref{appendix:linf} we recall some basic results for the reader's convenience).

%+------------------------------------------------+
\subsection{Higher Courant algebroids}
%+------------------------------------------------+
    Courant algebroids are used to simultaneously generalize Lie algebroids, Poisson, and symplectic structures (via Dirac structures).
    Beyond their intrinsic geometric interest, Courant algebroids provide the natural framework for generalized complex geometry in the sense of Hitchin and Gualtieri, and have found applications ranging from constrained and nonholonomic mechanics to duality symmetries in string theory.
    The standard Courant algebroid \(E^{(1)}=TM\oplus T^*M\) can be seen as the prototypical instance of a richer hierarchy of structures called \emph{higher Courant algebroids}, which are used in \cite{Zambon2012} to provide a higher analog of Dirac structures.

    \begin{definition}[Higher Courant Algebroid]\label{def.higher_courant_algebroid}
        Let $M$ be a smooth manifold, $r\geq 0$ an integer and $\sigma\in \Omega^{r+2}_{\mathrm{cl}}(M)$ a closed differential form. 
        A ($\sigma$-twisted)\emph{higher Courant algebroid} (of degree $r$) on $M$ consists of the data
        \begin{displaymath}
            \bigl(E^{(r)},\,\rho,\,\langle\cdot,\cdot\rangle_{\pm},\, [\cdot,\cdot]_{\sigma}\bigr)~,
        \end{displaymath}
        where:
        \begin{itemize}
            \item $E^{(r)}=TM\oplus \wedge^r T^*M$ is a vector bundle over $M$;
            
            \item $\rho: E^{(r)} \to TM$ is the anchor map, given by $\rho \binom{x}{\alpha}= x$ for any $\binom{x}{\alpha}\in \Gamma(E^{(r)})$; 
            
            \item $\langle\cdot,\cdot\rangle_{\pm}: E^{(r)}\otimes E^{(r)}\to \wedge^{r-1}T^*M$ is the (skew-)symmetric pairings, given by 
            \begin{displaymath}
                \left\langle\binom{x_1}{\alpha_1},\binom{x_2}{\alpha_2}\right\rangle_{\pm}=
\tfrac{1}{2}\bigl(\iota_{x_1}\alpha_2 \pm \iota_{x_2}\alpha_1\bigr)~,
                \qquad \forall \binom{x_i}{\alpha_i}\in (E^{(r)})_p~;
            \end{displaymath}
            \item $[\cdot,\cdot]_{\sigma}:\Gamma(E^{(r)})\otimes \Gamma(E^{(r)})\to \Gamma(E^{(r)})$ is the $\sigma$-twisted higher Courant bracket, given by 
            \begin{displaymath}
                \left[\binom{x_1}{\alpha_1},\binom{x_2}{\alpha_2}\right]_{\sigma}
                =
                \binom{[x_1,x_2]}{
            \Lie_{x_1}\alpha_2 - \Lie_{x_2}\alpha_1
- \ddR\,\big\langle \binom{x_1}{\alpha_1},\binom{x_2}{\alpha_2}\big\rangle_{-}
+ \iota_{x_1}\iota_{x_2}\sigma }~,
                \qquad \forall \binom{x_i}{\alpha_i}\in \Gamma(E^{(r)})~;
            \end{displaymath}
        \end{itemize}
    \end{definition}

    \begin{remark}
        Notice that for \(r=0\) one recovers the standard (twisted) Lie algebroid on \(TM\oplus \wedge^0T^*M\cong TM\oplus\mathbb R\);
        for \(r=1\) the standard (twisted) exact Courant algebroid on \(TM\oplus T^*M\).
    Using the Cartan formula $\Lie_X\alpha=\d(\iota_X\alpha)+\iota_X\d\alpha$, the bracket can also be written as
    $$
        \left[\binom{x_1}{\alpha_1},\binom{x_2}{\alpha_2}\right]_\sigma
        =\binom{[x_1,x_2]}{\ddR\,\big\langle \binom{x_1}{\alpha_1},\binom{x_2}{\alpha_2}\big\rangle_{-}
        + \iota_{x_1}\ddR\,\alpha_2 - \iota_{x_2}\ddR\,\alpha_1
        + \iota_{x_1}\iota_{x_2}\sigma}~.
    $$
    \end{remark}

    \begin{remark}[Naming and generalization]
        We notice that in \cite{Ritter2015a} (see also \cite{Miti2024}), the higher Courant algebroids introduced in \autoref{def.higher_courant_algebroid} are referred to as \emph{Vinogradov algebroids}.  
        Moreover, similarly to how the standard Lie algebroid can be regarded as a particular instance of a generic split Lie algebroid (see \cite{Liu1997}), one may seek a more abstract formulation of higher Courant algebroids for which \autoref{def.higher_courant_algebroid} represents only a specific realization of a prototypical standard object. A fully general definition in this direction was proposed in~\cite[Def.~5.6]{Grutzmann2015}.
    \end{remark}

%+------------------------------------------------+
\subsubsection{Higher Courant \texorpdfstring{$L_\infty$}{Lie infinity}-algebras}
%+------------------------------------------------+
%
In this subsection, we recall Zambon's graded-geometric construction\cite{Zambon2012}, which associates a canonical differential graded Lie algebra and the corresponding $L_\infty$-algebra to a higher Courant algebroid, both naturally encoded in the symplectic geometry of a suitable shifted cotangent.

Consider the graded manifold $T^*[r]T[1]M$, where $M$ is an ordinary finite-dimensional smooth manifold of dimension $d$ and $r$ is an arbitrary integer.
It is a standard result in graded geometry (see e.g. \cite{Cueca2021} for a recent overview) that the associated algebra of smooth functions $\mathcal{C}(T^*[r]T[1]M)$ is naturally a (shifted) $r$-Poisson algebra \cite{Cattaneo2006}.
After shifting by $r$ the $r$-Poisson algebra $\mathcal{C}(T^*[r]T[1]M)[r]$, one gets a graded Lie algebra $\gLarge_{r}$.
The GLA $\gLarge_{r}$ is equipped with a distinguished Maurer–Cartan element $\mathcal{S}$ lifting the de Rham differential on $M$.
%
%\begin{remark}[On the choice of the Maurer--Cartan element $\mathcal{S}$]\label{rmk:Sfunction-Twist}
%    The Maurer--Cartan element $\mathcal{S} \in \gLarge$ of \eqref{eq:canonical-MC-element-S} is not chosen arbitrarily but it encodes canonically the deRham differential on $M$ as follows.
More precisely, the natural projection of graded vector bundles over $M$,
    $$
        \pi:~~ T^*[r]T[1]M \twoheadrightarrow T[1]M ~,
    $$
induces, at the level of corresponding algebras of functions, an injective morphism of graded associative algebras
    $$
        \pi^*:~~ \mathcal{C}(T[1]M) \hookrightarrow \mathcal{C}(T^*[r]T[1]M)~.
    $$
It is a standard result in graded geometry that the algebra $\mathcal{C}(T[1]M)$ canonically identifies with the algebra of differential forms $\Omega(M)$ on $M$ (see, e.g.,~[\S 3]\cite{Qiu2011} or \cite{Fairon2017} for an elementary introduction to graded geometry), and one has
    \begin{equation*}
        \lbrace \mathcal{S}, \pi^*(\omega) \rbrace = \pi^*(\ddR\omega)~, \qquad \forall \omega \in \Omega^k(M) ~.
    \end{equation*}

\begin{definition}[The canonical DGLA associated with the standard higher Courant algebroid]
    \label{def:canonical-DGLA-higher-courant}
    Let $M$ be a smooth manifold and $r\geq 0$ an integer.
    The \emph{canonical differential graded Lie algebra associated with the standard higher Courant algebroid on $M$} is the differential graded Lie algebra 
    \begin{equation*}
        \gLarge_{r} = \mathcal{C}(T^*[r]T[1]M)[r]
    \end{equation*}
        obtained by endowing the graded Lie algebra arising from the $r$-desuspension of the $r$-Poisson algebra $\mathcal{C}(T^*[r]T[1]M)$  with the differential $\d=\lbrace\mathcal{S},\cdot \rbrace$.

\end{definition}

    In the following, we will canonically identify differential $k$-forms on $M$ as degree $r-k$ elements in $\gLarge_r$, dropping the symbol $\pi^*$ with a slight abuse of notation. Accordingly, since the differential $\d$ in $\gLarge_r$ is precisely given by the adjoint of $\mathcal{S}$ w.r.t the bracket with $\mathcal{S}$, we can write $$\d \omega = \ddR \omega$$ without ambiguities.
    
    Finally, notice that any differential $(r$+$1)$-form $\omega \in \Omega^{r+1}(M)$ is a degree one element in $\gLarge_r$ and trivially satisfies $\{\omega,\omega\}=0$. 
    Therefore the equation $\ddR \omega =0$ is equivalent to the Maurer–Cartan equation
    \begin{displaymath}
        \d \omega + \frac{1}{2}\{\omega, \omega\} = 0 ~.
    \end{displaymath}
    This means that a differential form $\omega \in \Omega^{r+1}(M)$ is a Maurer–Cartan element in $\gLarge_r$ precisely when $\omega$ is closed.
%\end{remark}
%
As a consequence, any closed differential form $\sigma \in \Omega^{r+1}(M)$ can be used to define a twisted version of \autoref{def:canonical-DGLA-higher-courant}.
\begin{definition}[Canonical DGLA associated with the twisted higher Courant algebroid]
    \label{def:twisted-canonical-DGLA-higher-courant}
    Let $M$ be a smooth manifold, $r\geq 0$ an integer and $\sigma \in \Omega^{r+1}_{\mathrm{cl}}(M)$ a closed differential form.
    The \emph{$\sigma$-twisted canonical differential graded Lie algebra associated with the higher Courant algebroid on $M$} is the differential graded Lie algebra 
    \begin{equation*} 
        \gLarge_{r,\sigma} = \mathcal{C}(T^*[r]T[1]M)[r]
    \end{equation*}
    obtained by shifting the $r$-Poisson algebra $\mathcal{C}(T^*[r]T[1]M)$ and equipping it with the \emph{twisted differential}
    \begin{equation*}
        \d_\sigma = \{\mathcal{S} + \sigma, \blank\}~.
    \end{equation*}
\end{definition}

In \cite{Zambon2012}, with each higher Courant algebroid of degree $r$+$1$ twisted by a closed form $\sigma \in \Omega_{\mathrm{cl}}^{r+1}(M)$ (possibly $\sigma = 0$), Zambon associates an $L_\infty$-algebra, here denoted by $\Courant[\sigma]{\,r-1}$. 
Zambon's construction builds upon the work of Getzler \cite{Getzler2010}, who showed that the truncation in negative degrees $\g^{<0}$ of any differential graded Lie algebra $\g$ carries a canonical $L_\infty[1]$-structure induced by derived brackets. 
\begin{definition}[Higher Courant $L_\infty$-algebra {\cite{Zambon2012}}]
    \label{def:higher-courant-L-infinity}
    Let $M$ be a smooth manifold, $r\geq 1$ an integer and $\sigma \in \Omega^{r+1}_{\mathrm{cl}}(M)$ a closed differential form.
    The \emph{(degree $r{-}1$) $\sigma$-twisted higher Courant $L_\infty$-algebra} $\Courant[\sigma]{\,r{-}1}$ is given by the chain complex
            \begin{equation*}
                \Omega^0(M) \xrightarrow{\d_\sigma} \Omega^1(M) \xrightarrow{\ddR} \cdots \xrightarrow{\ddR} \Omega^{r-2}(M) \xrightarrow{(0,\ddR)} \X(M) \oplus \Omega^{r{-}1}(M)
            ~,
        \end{equation*}
         where the first item is in degree $(1-r)$, the last one is in degree $0$, and with $\mathfrak{X}(M)$ being the space of vector fields on $M$;
         Together with a collection of $r$ skew-symmetric multibrackets $\{\blank,\dots,\blank\}_{k}$ of arity $k$ for $2 \leq k \leq r+1$
         induced by Getzler's derived bracket construction \cite{Getzler2010} on the DGLA $\gLarge_{r,\sigma}$ (see \cite{Zambon2012,Miti2024} for an explicit description).
\end{definition}

Notice that the underlying graded vector space of $\Courant[\sigma]{\,r-1}$ coincides with the desuspension of the negatively graded truncation $\gLarge_{r,\sigma}^{<0}$ of $\gLarge_{r,\sigma}$. %associated with the higher Courant algebroid of degree $r{-}1$ twisted by $\sigma$, 
%I.e, it is essentially a shifted truncation of the deRham complex.
 %
As Getzler already remarks in his preprint \cite{Getzler2010}, the procedure inducing derived brackets on the negative truncation $\g^{<0}$ of a DGLA $\g$ can be understood as a particular instance of the \emph{cone construction} introduced in \cite{Fiorenza2007} applied to the inclusion $\g^{\ge 0} \hookrightarrow \g$,
via the identification  $\g^{<0} \cong \g / \g^{\ge 0}$ (see \cite[Ex. 6.17]{Bandiera2015} for the whole argument in terms of homotopy limits). 
See \cite{Pridham2010} for a proof that the $L_\infty$-algebra associated with the cone construction of a morphism of DGLAs $f:\g \to \h$ indeed models the homotopy fiber $\hoFib(f)$. 
In other words, we have the following:
    \begin{proposition}[Higher Courant $L_\infty$-algebra as homotopy fiber]
        \label{prop:higher-courant-as-hofib}
        Let $M$ be a smooth manifold, $r\geq 1$ an integer and $\sigma \in \Omega^{r+1}_{\mathrm{cl}}(M)$ a closed differential form.
        The \emph{$\sigma$-twisted higher Courant $L_\infty$-algebra of degree $r{-}1$} is a \emph{model for the homotopy fiber} of the canonical inclusion of DGLAs $\gLarge_{r,\sigma}^{\ge 0} \hookrightarrow \gLarge_{r,\sigma}$, i.e.
        \begin{equation*}
            \Courant[\sigma]{\,r{-}1} \cong \hoFib\big(\gLarge_{r,\sigma}^{\ge 0} \hookrightarrow \gLarge_{r,\sigma}\big)~,
        \end{equation*}
        where $\gLarge_{r,\sigma}^{\ge 0}$ is the non-negatively graded truncation of the twisted canonical DGLA $\gLarge_{r,\sigma}$ associated with the higher Courant algebroid of degree $r{-}1$ twisted by $\sigma$ as per \autoref{def:canonical-DGLA-higher-courant}.
    \end{proposition}

In \autoref{sec:smallg} we describe an alternative derivation of the $L_\infty$-algebra $\Courant[\sigma]{\,r{-}1}$ starting from a simpler and more tractable DGLA, which simplifies explicit computations and highlights the role played by Cartan calculus.
The proof relies on the following technical lemma.

\begin{lemma}\label{lemma:Getzler-fibers-isomorphism}
    Let $\g$ be a DGLA and let $\h \hookrightarrow \g$ a sub-DGLA such that $\h^{<0} = \g^{<0}$ as graded vector spaces. 
    %and the following diagram holds at the level of the underlying graded vector spaces
    Then the commutative diagram of DGLAs
    \begin{equation}\label{eq:diagramma-h-g}
        \begin{tikzcd}
            \h^{\geq 0} \ar[r,hookrightarrow] \ar[d,hookrightarrow] & \h  \ar[d,hookrightarrow] % \ar[r,twoheadrightarrow]
            %& \frac{\h}{\h_{\geq 0}}\cong \h_{<0} \ar[d,equal] 
            \\
            \g^{\geq 0} \ar[r,hookrightarrow] & \g %\ar[r,twoheadrightarrow] & \frac{\g}{\g_{\geq 0}}\cong \g_{<0}        
        \end{tikzcd}
    \end{equation}
    induces an equivalence of $L_\infty$-algebras 
    $$\GetzFib[\g] \cong \GetzFib[\h]~.$$ 
    %of the inclusions $\g_{\geq 0} \hookrightarrow \g$ and $\h_{\geq 0} \hookrightarrow \h$ are isomorphic as $L_\infty[1]$-algebras.
\end{lemma}
\begin{proof}
    An explicit model for the homotopy fiber $\GetzFib[\g]$ is given in \cite[Thm. 3]{Getzler2010} and in our cohomological conventions reads as follows.
    As a chain complex, $\hoFib(\g^{\geq 0} \hookrightarrow \g)$ is equivalent to $(\g / \g^{\geq 0})[1] \cong \g^{<0}[1]$,
    where the differential on $\g^{<0}$ is given by the truncated differential
    \begin{displaymath}
        \morphism{\{\blank\}_1}{\g^{<0}}{\g^{<0}[1]}{a}
        {
            \begin{cases}
            \d a & \text{if } |a| < -1 ~, \\
            0 & \text{if } |a| \geq -1~.
            \end{cases}
        }
    \end{displaymath}
    %
    %Let be $L$ the graded vector space underlying the DGLA $\g$. Denote by $L_{\geq 0}$, resp. $L_{<0}$ the non-negatively , resp. positive, graded truncation of $L$. $L_{\geq 0}$ is a sub-DGLA of $\g$ and $L_{<0}$ is isomorphic, as a graded vector space, to the quotient $L/L_{\geq 0}$.
    The chain complex $\g^{<0}$ carries an $L_\infty[1]$-algebra structure with multibrackets given by derived brackets as follows.
    For $n=1$, the unary bracket is given by the truncated differential above.
    For $n>1$, the $n$-ary multibracket is given by
    \begin{displaymath}
        \morphism{\{\blank,\dots,\blank\}_n}{(\g^{<0})^{\otimes n}}{\g^{<0}[1]}{a_1\otimes \cdots \otimes a_n}
        {
            \displaystyle
            b_n \sum_{\sigma \in S_{n+1}}(-1)^{\varepsilon(\sigma)}
            [[\dots[[Da_{\sigma_1},a_{\sigma_1}],a_{\sigma_2}],\dots],a_{\sigma_n}]
        }
    \end{displaymath}
    where $D$ is the operator defined as
    \begin{displaymath}
        \morphism{D}{\g}{\g[1]}{a}{\begin{cases}
            \d a & \text{if } |a| < 0 ~, \\
            0 & \text{if } |a| \geq 0~,
        \end{cases}}
    \end{displaymath}
    the coefficient $b_n=\frac{(-1)^n B_n}{n!}$ is proportional to the  the Bernoulli numbers $B_n$,
    and the sign $(-1)^{\varepsilon(\sigma)}$ is the Koszul sign associated with the permutation $\sigma \in S_{n+1}$ acting on the graded tensor product $a_1\otimes \cdots \otimes a_n$.
    \\
    By the explicit formulas above, one directly sees that the $L_\infty[1]$-algebra structure on $\h^{<0}$ coincides with the one on $\g^{<0}$.
%    Now, consider two DGLAs $\g$ and $\h$ such that $\h\hookrightarrow \g$  and $\h_{<0} \equiv \g_{<0}$ as graded vector spaces.
 %   Then, the above construction applied to $\g$ and $\h$ yields two identical $L_\infty[1]$-algebras structures on the same underlying graded vector space $\g_{<0} \equiv \h_{<0}$ defined by the same derived brackets.
\end{proof}

%+------------------------------------------------+
\subsection{Multisymplectic Observables \texorpdfstring{$L_\infty$}{L-infinity}-algebra}
%+------------------------------------------------+
Given a smooth manifold $M$, a differential form $\sigma \in \Omega^{r+1}(M)$ is said to be \emph{multisymplectic} (or $r$-plectic) if it is closed and non-degenerate\footnote{In this context, non-degenerate means that the bundle map $\iota_{\blank}\sigma:TM \to \wedge^r T^*M$ is injective.} when the latter condition is dropped, one speaks of a \emph{pre-multisymplectic structure} \cite{Cantrijn1999}. 
The motivation for studying such geometric structures originates from the search for a purely geometric formulation of classical field theory \cite{Gotay1997,Gotay2004}. As noted, for instance, in \cite{Forger2015}, one of the main obstacles to developing this framework has been the difficulty of identifying a suitable analogue of the Poisson algebra of observables for a multisymplectic phase space. 
In the early 2010s, Rogers proposed a natural candidate for such an algebra by generalizing the classical Poisson algebra to an $L_\infty$-algebra of observables \cite{Rogers2012}. Subsequently, Rogers, together with Fiorenza and Schreiber, \cite{Fiorenza2014} observed that this construction admits a homotopical interpretation: it can be understood as a model for the homotopy fiber of the $L_\infty$-morphism given by contractio with $\sigma$. 
\\
This section is devoted to recalling these constructions and establishing the necessary background to relate them to the higher Courant framework.
The starting point consists of noticing that prescribing a preferred closed differential form $\sigma \in \Omega^{r+1}_{cl}(M)$ yields a criterion for selecting a distinguished class of vector fields on $M$; namely, the Hamiltonian ones.
\begin{definition}[Hamiltonian pairs]\label{def:hamiltonian-pairs}
    Let $M$ be a smooth manifold and $\sigma \in \Omega^{r+1}_{cl}(M)$ a closed differential form of degree $r+1$.
    We call a \emph{Hamiltonian vector field} any vector field $X \in \mathfrak{X}(M)$ such that the contraction $\iota_X \sigma$ is an exact $r$-form and denote by 
    \[
    \mathfrak X_{\sigma,\mathrm{ham}}   := \{\, X \in \mathfrak X \mid \iota_X \sigma \text{ is exact} \,\}
    \]
    the vector space of all Hamiltonian vector fields on $M$.
    We call a \emph{Hamiltonian pair} any pair $(X,\alpha) \in \mathfrak{X}(M) \oplus \Omega^{r-1}(M)$ such that $\iota_X \sigma + \ddR \alpha = 0$ and denote by 
    \[
    \Ham^{(r-1)}_\sigma
    := \{\, (X,\alpha) \mid \iota_X \sigma + \mathrm d \alpha = 0 \,\}
    \subseteq \mathfrak X \oplus \Omega^{r-1}
    = \Gamma(E^{(r-1)}).
    \]
    the vector space of all Hamiltonian pairs on $M$.
\end{definition}

\begin{remark}%[On the non-degeneracy condition] 
    Hamiltonian vector fields form a Lie subalgebra of $\X(M)$ since
    $$
    \iota_{[X,Y]}\sigma = \mathcal L_X(\iota_Y\sigma) - \iota_Y(\mathcal L_X\sigma)~.
    $$ 
    Notice that a non-degeneracy condition on $\sigma$ is not required to obtain a Lie algebra: it suffices that a primitive function exists. 
    %However, non-degeneracy ensures that the projection $\Ham^{(r)}_\sigma \to \X_{\sigma,\mathrm{ham}}$ is an isomorphism, i.e., that each Hamiltonian vector field admits a unique Hamiltonian form up to a closed form.
\end{remark}
    In the symplectic case ($r=1$), non-degeneracy is crucial to ensure that the space of Hamiltonian functions $C^\infty(M)$ is closed under the Poisson bracket, making it a Poisson algebra and so in particular a Lie algebra.
    In the multisymplectic case ($r>1$), the situation is more subtle: the space of Hamiltonian pairs is not a Lie algebra under any natural bracket operation, and one must resort to higher brackets to obtain an $L_\infty$-algebra structure.

    \begin{definition}[Multisymplectic observables \texorpdfstring{$L_\infty$}{L-infinity}-algebra {\cite{Rogers2012}}]
        Let $M$ be a smooth manifold and $\sigma \in \Omega^{r+1}_{cl}(M)$ a closed differential form of degree $r+1$.\\
        The \emph{(Rogers') multisymplectic observables $L_\infty$-algebra} $\Rogers[\sigma]{r-1}$ is given by chain complex
        \begin{equation*}
            \Omega^0(M) \xrightarrow{\ddR} \Omega^1(M) \xrightarrow{\ddR} \cdots \xrightarrow{\ddR} \Omega^{r-2}(M) \xrightarrow{(0,\ddR)} \Ham^{(r-1)}_\sigma,
        \end{equation*}
        where the first item is in degree $-r+1$, the last one is in degree $0$,
        and whose only non-trivial multibrackets are given by:
        %Together with a collection of $r$ skew-symmetric multibrackets $\{\blank,\dots,\blank\}_{k}$ of arity $k$ for $2 \leq k \leq r+1$ given by
\begin{equation}\label{eq:rogers-multibrackets}
	\begin{tikzcd}[column sep= small,row sep=0ex]
		\{\blank,\dots,\blank\}_k ~\colon&[-1em] \left(\Ham^{(r-1)}_\sigma\right)^{\otimes k} 	\arrow[r]& 				\Omega^{r+1-k}(M) \\[-1ex]
		& \binom{x_1}{\sigma_1}\otimes\dots\otimes \binom{x_k}{\sigma_k} 	\ar[r, mapsto]& 	(-1)^{k+1}
		\iota_{x_1}\cdots\iota_{x_k}\sigma~. 
	\end{tikzcd}
\end{equation}
\end{definition}
\noindent
The verification that the above brackets satisfy the $L_\infty$-algebra identities can be found in \cite[Thm. 5.2]{Rogers2012} (see also \cite[Lem. 2.3.9]{Miti2021}).

A crucial observation is that, like Zambon's Courant $L_\infty$-algebra, also the multisymplectic observables $L_\infty$-algebra can be understood as the homotopy fiber of a suitable $L_\infty$-morphism between DGLAs \cite[Thm. 3.12]{Fiorenza2014}.
\begin{proposition}[Multisymplectic observables as a homotopy fiber]
    The multisymplectic observables $L_\infty$-algebra $\Rogers[\sigma]{r-1}$ is a model for the homotopy fiber of the $L_\infty$-morphism %(see \cite[Prop. 3.8]{Fiorenza2014})
    $\iota_\infty \sigma: \X_{\sigma,\mathrm{ham}}(M) \to \Trunc\left(\Omega(M)[n]\right)$, given for any $k\geq 1$ by
    \begin{equation}\label{eq:iota-infinity-sigma}
        \morphism{(\iota_\infty\sigma)^k}{\big(\X_{\sigma,\mathrm{ham}}(M)\big)^{\otimes k}}{\Omega^{r+1-k}(M)}{x_1 \otimes \cdots \otimes x_k}
        {-\iota_{x_1} \cdots \iota_{x_k} \sigma}~,
%        ~,
    \end{equation}
    where $\X_{\sigma,\mathrm{ham}}(M)$ is the Lie algebra of Hamiltonian vector fields on $M$ %(see \autoref{def:hamiltonian-pairs})
    and $\Trunc\left(\Omega(M)[n]\right)$ is the abelian DGLA (i.e., the chain complex) given by the canonical truncation of the $n$-shifted de Rham complex
    \begin{displaymath}
        \Omega^0(M) \xrightarrow{\ddR} \Omega^1(M) \xrightarrow{\ddR} \cdots \xrightarrow{\ddR} \Omega^{n-1}(M) \xrightarrow{\ddR} \ddR\Omega^{n-1}(M)~,
    \end{displaymath}
    with $\ddR\Omega^n(M)$ in degree $0$. 
\end{proposition}
\begin{remark}[On sign conventions]
    The apparent discrepancy of signs in \eqref{eq:rogers-multibrackets} with \cite{Callies2016} and \eqref{eq:iota-infinity-sigma} with \cite{Fiorenza2014} is due to different conventions adopted regarding the multicontraction operator,
    since $\iota_{x_1} \cdots \iota_{x_k} \sigma = (-1)^{\frac{k(k-1)}{2}}\iota_{x_k} \cdots \iota_{x_1} \sigma$.
    % and $\sum_{i=1}^{k-1}i = \frac{k(k-1)}{2}$.
\end{remark}
    %\antonio{In Callies-Fregier-Rogers-Zambon si definisce come $\varsigma(k) \iota_{x_k}\cdots \iota_{x_1}\sigma$ con $\varsigma(k) = -(-1)^{\frac{k(k+1)}{2}}=-(-1^{k + \frac{k(k-1)}{2}})=(-1)^{k+1}(-1)^{\frac{k(k-1)}{2}}$}
    %\antonio{in Fiorenza-Rogers Schreiber è definito come ${-(-1)^{\frac{k(k-1)}{2}}\iota_{x_k} \cdots \iota_{x_1} \sigma} $ ma $\iota_{x_1} \cdots \iota_{x_k} \sigma = (-1)^{\sum_{i=1}^{k-1}i}\iota_{x_k} \cdots \iota_{x_1} \sigma$ e $\sum_{i=1}^{k-1}i = \frac{k(k-1)}{2}$}

    In conclusion, both the multisymplectic observables $L_\infty$-algebra and the higher Courant $L_\infty$-algebra can be understood as models for homotopy fibers of suitable morphisms of differential graded Lie algebras: 
    \begin{align*}
        \Rogers[\sigma]{r-1} &\cong \hoFib\big(\iota_\infty\sigma)%: \X_{\sigma,\mathrm{ham}}(M) \to \Trunc\left(\Omega(M)[r]\right)\big)
        ~,\\
        \Courant[\sigma]{\,r-1} &\cong \hoFib\big(\gLarge_{r,\sigma}^{\ge 0} \hookrightarrow \gLarge_{r,\sigma}\big)~.
    \end{align*}
    Homotopy fibers are a special case of homotopy limits, and as such, they enjoy a universal property. In the sequel, we will make use of this fact to construct a morphism between the twisted and untwisted higher Courant $L_\infty$-algebras and between the multisymplectic observables $L_\infty$-algebra and the higher twisted Courant $L_\infty$-algebra.

%=================================================#
\section{A morphism between the twisted and untwisted higher Courant \texorpdfstring{$L_\infty$}{L-infinity}-algebras}
%=================================================#
For any $r>0$ one has a map of cochain complexes $\Phi_1\colon \Courant[\sigma]{\,r-1} \to \Courant{r}$ defined by:
\begin{equation}\label{eq:courant-linear}
\begin{tikzcd}[column sep=small]
\Courant[\sigma]{\,r-1} \ar[r,phantom,":="] \ar[d,"\Phi_1"'] &
0 \ar[r] \ar[d] &
\Omega^0 \ar[d,"\d"] \ar[r] &
\Omega^1 \ar[d,"\d"] \ar[r] &
\cdots \ar[r] &
\Omega^{r-2} \ar[d,"\d"] \ar[r,"{(\d,0)}"] &
\Omega^{r-1}\oplus \X \ar[d,"{(\d,\mathrm{id})}"] \\
\Courant{r} \ar[r,phantom,":="] &
\Omega^0 \ar[r] &
\Omega^1 \ar[r] &
\Omega^2 \ar[r] &
\cdots \ar[r] &
\Omega^{r-1} \ar[r,"{(\d,0)}"] &
\Omega^{r}\oplus \X
\end{tikzcd}
\end{equation}
In \cite{Zambon2012}, it has been noticed that for $r=1$, the map $\Phi_1$ is the the linear part of an \(L_\infty\)-morphism \(\Phi\colon \Courant[\sigma]{0}\to \Courant[]{1}\).
The goal of this subsection is to prove that this is true for any $r>0$.
%such a chain map is indeed the linear part of an \(L_\infty\)-morphism between the corresponding higher Courant \(L_\infty\)-algebras $\Courant[\sigma]{\,r-1} \longrightarrow \Courant{r}$.
To this end, we will exploit the universal property of homotopy fibers together with an equivalent description of the higher Courant \(L_\infty\)-algebra based on a smaller and more manageable differential graded Lie algebra.
% than the canonical one introduced in \autoref{def:canonical-DGLA-higher-courant}.

%+------------------------------------------------+
\subsection{The Cartan DGLA and a smaller model for the higher Courant \texorpdfstring{$L_\infty$}{L-infinity}-algebra}
\label{sec:smallg}
%+------------------------------------------------+
%
In this subsection, we introduce a more manageable DGLA $\gSmall_{r,\sigma}$, which is a subDGLA of $\gLarge_{r,\sigma}$ such that the homotopy fiber of $\gSmall_{r,\sigma}\hookrightarrow \gSmall_{r,\sigma}$ is still a model for $\Courant[\sigma]{\,r-1}$.

We denote by $\Omega$ the differential graded commutative associative algebra of differential forms on $M$, equipped with the usual wedge product $\wedge$ and the de Rham differential $\ddR$, and by $\X$ the Lie algebra of vector fields on $M$ and $[\blank,\blank]_\X$ the corresponding Lie bracket. 
Also, we denote by  $\Der(\Omega)$ the DGLA of derivations of $\Omega$, with Lie bracket given by the graded commutator and differential given by 
 $[\ddR, \blank]$ since the de Rham operator on differential form is a Maurer–Cartan element in $\Der(\Omega)$.
    In the spirit of \cite{Alekseev2012}, we now introduce a differential graded Lie subalgebra of $\Der(\Omega)$ encoding the classical Cartan calculus on differential forms.

\begin{definition}[Cartan algebra]\label{def:CartanAlgebra}
    The \emph{Cartan algebra} is the differential graded Lie subalgebra $\Cartan \subset \Der(\Omega)$ generated by the contraction operators (or interior products):
    \begin{equation*}
        \iota_{\X} := \big\lbrace
            \iota_x : \Omega \to \Omega[-1] \,\big|\, x \in  \X
        \big\rbrace~.
    \end{equation*}
\end{definition}
\begin{lemma}
    The differential graded Lie algebra $\Cartan$ is isomorphic to the DGLA $\mathfrak{C}=(\mathfrak{X}[1] \oplus \mathfrak{X},\d_{\mathfrak{C}},[\blank,\blank]_{\mathfrak{C}})$ given by two copies of the Lie algebra of vector fields on $M$ (one in degree $0$ and one in degree $-1$), with the differential $d_{\mathfrak{C}}$ given by the shifted identity map $\id: \mathfrak{X}[1] \to \mathfrak{X}$ and Lie bracket $[-,-]_{\mathfrak{C}}$ given  by 
    \[
    \begin{cases}
        [x,y]_{\mathfrak{C}}=[x,y]\in \mathfrak{X} &\text{for any $x,y\in \mathfrak{X}$}\\
        [x,y_{[1]}]_{\mathfrak{C}} = ([x,y]_{[1]}) \in \mathfrak{X}[1] &\text{for any $x\in \mathfrak{X}$ and $y_{[1]}\in \mathfrak{X}[1]$}\\
        [x_{[1]},y_{[1]}]_{\mathfrak{C}} =0 &  \text{for any $x_{[1]},y_{[1]} \in \mathfrak{X}[1]$}
    \end{cases}
    \]
  \end{lemma}
\begin{proof}
    The Cartan algebra $\Cartan$ is generated by the set of contraction operators $\iota_\X$ (of degree $-1$) along vector fields, through the commutator $[\blank, \blank]$ and the differential $[\ddR, \blank]$. From equation \eqref{eq:CartanIotasquared}, it follows that all contraction operators $\iota_X$ commute with each other. Consequently, $\Cartan$ is concentrated in degrees greater than or equal to $-1$.  The degree $0$ part of $\Cartan$ is given by the image of $\iota_\X$ under $[\ddR, \blank]$. By Cartan's formula (equation \eqref{eq:CartanDvsIota}), we have $[\ddR,\iota_X] = \Lie_X$ for any $X \in \X$. Thus, the degree $0$ part of $\Cartan$ is:
    \begin{equation*}
        \Cartan^0 = \big\lbrace
            \Lie_X : \Omega \to \Omega \,\big|\, X \in \X(M)
        \big\rbrace~.
    \end{equation*}

    Equation \eqref{eq:CartanLievsD} ensures that the differential graded vector space underlying $\Cartan$ forms a short exact sequence concentrated in degrees $-1$ and $0$ within $\Der(\Omega)$:
    \begin{equation*}
        \begin{tikzcd}
            0 \arrow[r] & \iota_\X \arrow[rr, "{[\ddR,-]}"] && \Lie_\X \arrow[r] & 0~.
        \end{tikzcd}
    \end{equation*}
    The last two Cartan's identities (equations \eqref{eq:CartanLiesquared} and \eqref{eq:CartanLievsIota}) ensure that the above 2-term complex is indeed a differential graded Lie subalgebra of $\Der(\Omega)$. 

    By the very definition of the contraction $\iota_X$ and Lie derivative $\Lie_X$ with respect to the vector field $X \in \X$, it is clear that there exists a surjective map 
    \begin{equation*}
        \X[1]\oplus \X \to \iota_\X \oplus \Lie_\X~.
    \end{equation*}
    Injectivity follows from the fact that given two vector fields $X, Y \in \X$, the following equations are satisfied for any $\omega \in \Omega$ iff $X = Y$:
    \begin{align*}
        \iota_X(\omega) &= \iota_Y(\omega)~,\qquad
        \Lie_X(\omega) = \Lie_Y(\omega)~.
    \end{align*}
\end{proof}

\begin{remark}
    Notice that while $\X$, and so $\X\oplus\X[1]$, are $C^\infty(M)$-modules and $\Cartan$ is isomorphic to $\X[1] \oplus \X$ as a graded vector space, this isomorphism is not an isomorphism of $C^\infty(M)$-modules. 
    Specifically, for a smooth function $f \in C^\infty(M)$ and vector field $X \in \X$, we generally have $\Lie_{f\, X} \neq f \cdot \Lie_X$.
    Yet, the differential graded vector space $\Omega$ is naturally a dgLie module over $\Cartan$, and so over $\mathfrak{C}$, via the standard Cartan calculus operators; the same is true for the shifted dg-modules $\Omega[r]$, for any $r$.
\end{remark}

\begin{definition}[The DGLA $\gSmall_r$]
    The differential graded Lie algebra $\gSmall_r$ is given by the semidirect product of the Cartan dgLie algebra $\mathfrak{C}$. and the $\mathfrak{C}$-module of differential forms shifted by $r$:
    \begin{equation*}
        \gSmall_r = \mathfrak{C} \ltimes \Omega[r]~.
    \end{equation*}
    More explicitly it is given by the  cochain complex $\mathfrak{C}\oplus\Omega[r]$, i.e.
    $$
        \begin{tikzcd}[column sep = normal]
            \text{\tiny (deg $-r$)}  &  \cdots & \text{\tiny (deg $-2$)} & \text{\tiny (deg $-1$)}
            & \text{\tiny (deg $0$)} & \text{\tiny (deg $1$)} & \cdots
            \\[-2em]
            \Omega^0 \ar[r,"\ddR"]  &
            \cdots \ar[r,"\ddR"] & 
            \Omega^{r-2} \ar[r,"{\ddR}"] \ar[d,phantom,"\oplus"] & 
            \Omega^{r-1} \ar[d,phantom,"\oplus"] \ar[r,"\ddR"] & 
            \Omega^r\ar[d,phantom,"\oplus"]\ar[r,"\ddR"] & \Omega^{r+1}\ar[r,"\ddR"] \ar[d,phantom,"\oplus"]  & \cdots
            \\[-.8em]
            & &0 \ar[r,hook] &\X \ar[r,equal, "\mathrm{id}"] &\X \ar[r,two heads]& 0 &
        \end{tikzcd}~,
    $$
    together with the skewsymmetric bracket $\lbrace\blank,\blank\rbrace_{\gSmall_r}$ defined by
    \begin{align*}
        \lbrace{ X_{[1]}, \alpha}\rbrace_{\gSmall_r} &=~ \iota_X \alpha~= -(-1)^{|\alpha|} \lbrace{\alpha,X_{[1]}}\rbrace_{\gSmall_r}~,
        \\
        \lbrace{ X, \alpha}\rbrace_{\gSmall_r} &=~ \Lie_X \alpha~=-\lbrace{\alpha,X}\rbrace_{\gSmall_r}~,
        \\
        \lbrace{ X, Y}\rbrace_{\gSmall_r} &=~ {[X,Y]_\X}~=-\lbrace{Y,X}\rbrace_{\gSmall_r}~,
        \\
        \lbrace{ X, Y_{[1]}}\rbrace_{\gSmall_r} &=~ ({{[X,Y]_\X}})_{[1]}~=-\lbrace{Y_{[1]},X}\rbrace_{\gSmall_r}~,
    \end{align*}  
    for any $X, Y \in \X$ and $\alpha \in \Omega[r]$ with $|\alpha|=\deg(\alpha){-}r$, and zero in all other caaes.
\end{definition}
    In particular, any closed differential form $\sigma \in \Omega^{r+1}$ is a Maurer-Cartan element in $\gSmall_r$ and can therefore be used to twist the DGLA structure of $\gSmall_r$.
\begin{definition}[Twisted DGLA $\gSmall_{r,\sigma}$]\label{def:smallg-sigma}
   The \emph{$\sigma$-twisted DGLA} $\gSmall_{r,\sigma}$ is the differential graded Lie algebra obtained by twisting  $\gSmall_r$ by the Maurer-Cartan element $\sigma$. More explicitly, it is the same Lie bracket as $\gSmall_{r,\sigma}$, while the twisted differential  $\d_\sigma$ is given by $ \d_\sigma=d_{\gSmall_r}+ \{\sigma, \blank\}_{\gSmall_r}$. 
   That is,
       \begin{align*}
        \d_\sigma X_{[1]} &= X + \iota_X \sigma~,
        \\
        \d_\sigma X &= -\Lie_X \sigma~,
        \\
        \d_\sigma \alpha &= \ddR \alpha ~,
    \end{align*}
    for any $X \in \X$ and $\alpha \in \Omega[r]$.

\end{definition}

\begin{lemma}\label{lemma:gsmall-into-glarge}
    There is a canonical inclusion of differential graded Lie algebras $\gSmall_{r,\sigma} \hookrightarrow \gLarge_{r,\sigma}$.
\end{lemma}
\begin{proof}
    The $r$-Poisson structure on $T^*[r]T[1]M$ can be conveniently described in terms of local bundle coordinates associated with local coordinates on the base manifold $M$.
    Let 
    \begin{itemize}
    \item $\{x^i\}_{1\leq i \leq d}$ be degree $0$ local coordinates on $M$,
    \item $\{v_i\}_{1\leq i \leq d}$ be degree $1$ local coordinates on the fibers of $T[1]M\to M$ naturally induced by $\{x^i\}$,
    \item $\{p_i\}_{1\leq i \leq d}$ be degree $r-1$ local coordinates on the fibers of $T^*[r]T[1]M\to T[1]M$ naturally dual (i.e. conjugated to) $\{v^i\}$,
    \item $\{P^i\}_{1\leq i \leq d}$ be degree $r$ local coordinates on the fibers of $T^*[r]T[1]M\to T[1]M$ naturally dual (i.e. conjugated to) $\{x^i\}$.
    \end{itemize}
    The associative algebra of smooth functions $\mathcal{C}(T^*[r]T[1]M)$ is locally isomorphic to the algebra of polynomials in the graded variables $p,P,v$ with coefficients in smooth functions of the variables $x$, i.e.
\begin{equation*}
    \mathcal{C}(T^*[r]T[1]M)\cong_{\text{loc.}} 
    \Big\langle
        f^{j_1,\cdots,j_m~;k_1,\cdots,k_n}_{i_1,\cdots,i_\ell}(x^1,\dots)v^{i_1}\dots v^{i_\ell}p_{j_1}\dots p_{j_m} P_{k_1}\cdots P_{k_n}
    \Big\rangle~.
\end{equation*}
%
%Since the graded manifold $T^*[r]T[1]M$ carries a canonical degree-$r$ symplectic structure, and t
The above coordinates are Darboux coordinates for the canonical $r$-Poisson structure on $\mathcal{C}(T^*[r]T[1]M)$:
% is equipped with a degree $r$ Poisson bracket $\{\cdot,\cdot\}$, which is completely determined on generators $\{x^i, v^j, p_k, P_\ell\}$ by:
\begin{align*}
    \{x^i, P_j\} &= \delta^i_j = -\{P_j, x^i\},\\ 
    \{v^i, p_j\} &= \delta^i_j = -(-1)^{r-1} \{p_j, v^i\},\\
    \text{all other brackets} &= 0~.
\end{align*}
%Such equations embody the conjugacy relation between the coordinates $\{x^i\}$ and $\{P_j\}$, and the coordinates $\{v^i\}$ and $\{p_j\}$.
%
The Maurer-Cartan element $\mathcal{S}$ is written in local coordinates as 
        \begin{equation}
            \label{eq:canonical-MC-element-S}
            \mathcal{S} = \sum_{i=1}^d v^i P_i
            ~,
        \end{equation}
and the compatibility condition with the de Rham differential, i.e.  
    $\lbrace \mathcal{S}, \pi^*(\omega) \rbrace = \pi^*(\ddR\omega)$
    % for any $\omega \in \Omega^k(M)$
    , is retrieved from the identity $\pi^*(\d x^i) = v^i$.
    %\footnote{To verify this, it suffices to check the equation on generators of $\Omega(M)$, i.e., on polynomial functions $p(x^i)$ and 1-forms $\d x^i$, and then extend it to arbitrary forms by using the graded Leibniz rule for both sides.}
    On a local chart, for any fixed $r\geq 0$, $\Cartan$ can be seen as the subalgebra of $\gLarge_{r}$ generated by the following two types of polynomials
    \begin{align*}
        \Big\langle
            X^i(x) p_i ~,
            X^i(x) P_i + \frac{\partial X^i}{\partial x^j}(x) v^j p_i
        \Big \rangle.
    \end{align*}  
 The identification is given by
 \begin{align*}
     \iota_X &\leftrightarrow  X^i(x) p_i~,\\
     \Lie_X &\leftrightarrow  X^i(x) P_i+\frac{\partial X^i}{\partial x^j}(x) v^j p_i~,
 \end{align*}
 for  any $X=X^i\partial_i$.
    Similarly,
$\gSmall_r$ can be seen as the subalgebra of $\gLarge_{r}$ generated by the following three types of polynomials
    \begin{align*}
        \Big\langle
            X^i(x) p_i ~,
            X^i(x) P_i + \frac{\partial X^i}{\partial x^j}(x) v^j p_i~, \alpha_Iv^I
        \Big \rangle.
    \end{align*}  
    In other words, we have the following explicit inclusion of DGLAs:   
    \begin{equation*}
            \morphism{\iota}{\gSmall_{r}}
            {\gLarge_{r}}
            {\iota_X + \Lie_Y + \alpha}
            {
            X^i(x) p_i + \left( Y^i(x) P_i + \frac{\partial Y^i}{\partial x^j}(x) v^j p_i \right) + \alpha_I v^I~,
            }
    \end{equation*}
    for any $X = X^i \partial_i$ and $Y= Y^i \partial_i$ in $\X$ and $\alpha = \alpha_I dx^I \in \Omega$.
    Finally, since the given $\sigma \in \Omega^{r+1}_{\mathrm{cl}}(M)$ is a Maurer–Cartan element in $\gSmall_r$, the above inclusion extends directly to the twisted case.
\end{proof}

\begin{proposition}\label{prop:smallg-and-largeg-same-fiber}
    The two dgLie algebras $\gSmall_r$ and $\gLarge_r$ determine the same homotopy fiber, i.e.
    \begin{align*}
        \hoFib(\gSmall_{r,\sigma}^{\geq 0} \hookrightarrow \gSmall_{r,\sigma}) \cong
        \hoFib(\gLarge_{r,\sigma}^{\geq 0} \hookrightarrow \gLarge_{r,\sigma})
        =: \Courant[\sigma]{\,r{-}1}
        ~.
    \end{align*}
\end{proposition}
\begin{proof}
    By \autoref{lemma:Getzler-fibers-isomorphism}, it suffices to verify that the two dg Lie algebras $\gSmall_{r,\sigma}$ and $\gLarge_{r,\sigma}$ agree in negative degrees. This is indeed the case, since in both instances the negative part is given by the same graded vector space:
    \[
        \bigl(\Omega^0[r] \oplus \Omega^1[r{-}1] \oplus \cdots \oplus \Omega^{r-1}[1]\bigr) \oplus \X[1].
    \]
\end{proof}    
In this way, we have proven \autoref{bigthm:smaller-model} from the Introduction.

\begin{remark}[Purely algebraic generalizations]\label{rem:outlook-algebraic-generalizations}
    A key outcome of the constructions developed above is the central role played by the Cartan calculus of differential forms and vector fields. 
    More precisely, the arguments only rely on the underlying \emph{algebraic} structure of the Cartan calculus, rather than on the specific features of the smooth geometric setting.
    It is therefore natural to expect that the entire picture admits a meaningful extension to more general algebraic contexts where a Cartan calculus is available.
    As a first step in this direction, one may consider an arbitrary Lie algebra $\g$ together with a $\g$-module $M$. 
    In this case, the associated Chevalley--Eilenberg complex $\CE(\g,M)$ naturally carries a Cartan calculus structure.
    Pushing this perspective further, one may work in the framework of $BV$-modules, where a Cartan calculus is likewise present and where a corresponding Rogers-type $L_\infty$-algebra has already been constructed in~\cite{Miti2025}.
    In this latter setting, the formalism also connects to areas beyond differential geometry, such as Hochschild cohomology and deformation quantization.
    See ~\cite{Kowalzig2022,Fiorenza2020} for an operadic framework for Cartan calculi in algebraic contexts providing a unifying perspective on these structures.
\end{remark}

%-------------------------------------------------+
\subsection{The canonical morphism between \texorpdfstring{$\Courant[]{\,r-1}$ and $\Courant[]{r}$}{higher Courant L-infinity-algebras}}
%-------------------------------------------------+
According to \autoref{def:higher-courant-L-infinity} and \autoref{prop:smallg-and-largeg-same-fiber}, one has that the $\sigma$-twisted higher Courant \(L_\infty\)-algebra \(\Courant[\sigma]{\,r-1}\) is a model for the homotopy fiber of the inclusion of the non-negative degree part into the full DGLA for the DGLA  \(\gSmall_{r,\sigma}\) of \autoref{def:smallg-sigma}:
$$
    \Courant[\sigma]{\,r-1} \cong \GetzFib[\gSmall_{r,\sigma}].
$$
Therefore, due to the universal property of homotopy fibers, any homotopy commutative square of $L_\infty$-morphisms 
\begin{displaymath}
   \begin{tikzcd}[column sep=large]
\gSmall_{r,\sigma}^{\ge 0} \ar[r,hook] \ar[d] & \gSmall_{r,\sigma} \ar[d] \\
\gSmall_{r+1}^{\ge 0} \ar[ur,Rightarrow,"h"] \ar[r,hook] & \gSmall_{r+1}.
\end{tikzcd} 
\end{displaymath}
 will induce a canonical (up to homotopy equivalence)  \(L_\infty\)-morphism between the homotopy fibers of the horizontal inclusions, i.e., between the higher Courant \(L_\infty\)-algebras \(\Courant[\sigma]{\,r-1}\) and \(\Courant{r}\). Clearly, a homotopy commutative diagram as above can be trivially obtained by choosing the vertical arrows to be the zero morphisms. This will correspond to the trivial morphism beween \(\Courant[\sigma]{\,r-1}\) and \(\Courant{r}\). As we are going to show, there is also a distinguished nontrivial diagram of this kind.  
% \antonio{Rephrasing: Un modo canonico per ottenere un morfismo $L_\infty$ tra le fibreo omotopiche $\GetzFib[\gSmall_{r,\sigma}]$ e $\GetzFib[\gSmall_{r+1}]$ è di esibire un quadrato che commuta a meno di un'omotopia fissata}
We begin by providing it in the special case $\sigma=0$. One has the following Lemma, whose proof is immediate.
\begin{lemma}\label{lemma:Phi1}
    Let $\gSmall_r$ and $\gSmall_{r+1}$ be the DGLAs introduced in \autoref{def:smallg-sigma}. Then the map $\Phi_1\colon \gSmall_r\to \gSmall_{r+1}$ defined by 
    \[
\begin{tikzcd}
0 \ar[r]\ar[d] &
\Omega^0 \ar[r,"\ddR"]\ar[d,"\ddR"] &
\Omega^1 \ar[r,"\ddR"]\ar[d,"\ddR"] &
\Omega^2 \ar[r,"\ddR"] &
\cdots \ar[r,"\ddR\oplus 0"] &
\Omega^{r-1}\oplus {\X} \ar[r,"\ddR\oplus \mathrm{id}"]\ar[d,"\ddR\oplus \mathrm{id}"] &
\Omega^{r}\oplus {\X} \ar[r,"\ddR+0"]\ar[d,"\ddR\oplus \mathrm{id}"] &
\Omega^{r+1} \ar[d,"\ddR"] \ar[r,"\ddR"] &\cdots
\\
\Omega^{0} \ar[r,"\ddR"] &
\Omega^{1} \ar[r,"\ddR"] &
\Omega^{2} \ar[r,"\ddR"] &
\cdots &
\cdots \ar[r,"\ddR\oplus 0"] &
\Omega^{r}\oplus {\X} \ar[r,"\ddR\oplus \mathrm{id}"] &
\Omega^{r+1}\oplus {\X} \ar[r,"\ddR+0"] &
\Omega^{r+2}\ar[r,"\ddR"]&\cdots
\end{tikzcd}
~,
\]
i.e., explicitly,
\begin{displaymath}
    \Phi_1(\omega) = \ddR\omega~, \quad
    \Phi_1(X_{[1]}) = X_{[1]}~, \quad
    \Phi_1(X) = X~,
\end{displaymath}
for any differential form $\omega \in \Omega$ and vector field $X \in \X$,
is a morphism of chain complexes.
\end{lemma}
The morphism of cochain complexes  $\Phi_1$ is not a DGLA morphism, since the de Rham differential does not generally commute with contractions:
    \begin{displaymath}
       \Phi_1\{X_{[1]},\omega\}_{\gSmall_{r}}\neq \{\Phi_1(X_{[1]}),\Phi_1(\omega)\}_{\gSmall_{r+1}}~,
    \end{displaymath}
 yet this can be cured, as shown by the following.   

\begin{lemma}\label{lem:Linfty-morphism-smallgr-to-smallgr+1}
   The morphism of cochain compelxes $\Phi_1\colon \gSmall_r\to \gSmall_{r+1}$ from Lemma \ref{lemma:Phi1}  is the linear component of an \(L_\infty\)-morphism \(\Phi: \gSmall_r \to \gSmall_{r+1}\) which 
  is compatible with   
 non-negative truncations, thus giving a strictly commutative diagram of $L_\infty$-morphisms
    \begin{displaymath}
        \begin{tikzcd}[column sep=large]
            \gSmall_r^{\ge 0} \ar[r,hook] \ar[d,"\Phi\vert_{{\ge 0}}"'] & \gSmall_r \ar[d,"\Phi"] \\
            \gSmall_{r+1}^{\ge 0} \ar[r,hook] & \gSmall_{r+1}.
        \end{tikzcd}
    \end{displaymath}
\end{lemma}
\begin{proof}
We define the binary component $    \Phi_2:\gSmall_r\otimes \gSmall_r \to \gSmall_{r+1}[1]$ as
\begin{displaymath}
    \Phi_2(X_{[1]},\omega)=\iota_X\omega,\qquad
    \Phi_2(\omega,X_{[1]})=-(-1)^{|\omega|} \iota_X\omega~,
\end{displaymath}
for any vector field \(X\in \X\) and differential form \(\omega\in \Omega\), and zero otherwise. Here $|\omega| = \deg(\omega)-r$ is the shifted degree of the differential form \(\omega\).
Then $\Phi=(\Phi_1,\Phi_2,0,0,\dots)$ is an $L_\infty$-morphism \(\g_r\to \g_{r+1}\). This reduces to checking the three identities \eqref{eq:Linfty-morphism-DGLAs-arity-1}, \eqref{eq:Linfty-morphism-DGLAs-arity-2} and \eqref{eq:Linfty-morphism-DGLAs-arity-3} from \autoref{rmk:Linfty-morphism-DGLAs}).
%    According to \autoref{def:Linfty-morphism-DGLAs} (see also \autoref{rmk:Linfty-morphism-DGLAs}), one can check the three identities \eqref{eq:Linfty-morphism-DGLAs-arity-1}, \eqref{eq:Linfty-morphism-DGLAs-arity-2} and \eqref{eq:Linfty-morphism-DGLAs-arity-3} 
 Equation    \eqref{eq:Linfty-morphism-DGLAs-arity-1} is the fact that $\Phi_1$ is a morphism of cochain complexes.
    by inspectiong on any combination of the entries with $X_i \in \X(M)$ and $\omega_j \in \Omega(M)$.
%    In the case of \eqref{eq:Linfty-morphism-DGLAs-arity-1}, one has 
%\begin{align*}
%v \colon\quad
%& \Phi_1 \circ \d_{\gSmall_{r}} (v)
%  = \d_{\gSmall_{r+1}} \circ \Phi_1 (v)
%\\[0.5em]
%\omega \colon\quad
%& \Phi_1(\ddR \omega)
%  = \d_{\gSmall_{r+1}} (\ddR \omega)
%\\[0.5em]
%\Lie_X \colon\quad
%& \Phi_1(0)
%  = \d_{\gSmall_{r+1}} (\Lie_X)
%\\[0.5em]
%\iota_X \colon\quad
%& \Phi_1(\Lie_X)
%  = \d_{\gSmall_{r+1}} (\iota_X)
%\end{align*}
%    which are all trivially checked.     
    Equation \eqref{eq:Linfty-morphism-DGLAs-arity-2} is the identity
    \[
    \Phi_1(\lbrace v_1,v_2\rbrace_{\gSmall_r}) - \lbrace \Phi_1(v_1),\Phi_1(v_2)\rbrace_{\gSmall_{r+1}} = \d_{\gSmall_{r+1}} \Phi_2(v_1,v_2) + \Phi_2(\d_{\gSmall_r} v_1,v_2) + (-1)^{|v_1|} \Phi_2(v_1,\d_{\gSmall_r} v_2)
    \]
    for any $v_1,v_2\in \gSmall_r$. Specializing the entries $v_i$ to be in  $\X$, in $\X[1]$, or in $\Omega$, and using the graded antisymmetry of the equation, we find the following set of identities to be satisfied:
\begin{align*}
(\omega_1,\omega_2) \colon\quad & \Phi_1(0) - \lbrace \ddR \omega_1,\ddR \omega_2\rbrace_{\gSmall_{r+1}} = 0
\\[0.5em]   
(X_{[1]},\omega) \colon\quad &
\Phi_1(\iota_X \omega) - \lbrace X_{[1]},\ddR \omega\rbrace_{\gSmall_{r+1}} = \d_{\gSmall_{r+1}} (\iota_X \omega) + (-1)^{|\iota_X|} \iota_X \ddR \omega
\\[0.5em]
(X_{[1]},Y_{[1]}) \colon\quad &
\Phi_1(0) - \lbrace X_{[1]},Y_{[1]}\rbrace_{\gSmall_{r+1}} = 0
\\[0.5em]
(X,\omega) \colon\quad &
\Phi_1(\Lie_X \omega) - \lbrace X,\ddR \omega\rbrace_{\gSmall_{r+1}} = 0
\\[0.5em]
(X,Y_{[1]}) \colon\quad &
\Phi_1\left(({[X,Y]_\X})_{[1]}\right) - \lbrace X,Y_{[1]}\rbrace_{\gSmall_{r+1}} = 0
\\[0.5em]
(X,Y) \colon\quad &
\Phi_1([X,Y]_\X) - \lbrace X,Y\rbrace_{\gSmall_{r+1}} = 0.
\end{align*}
These are all readily verified using Cartan calculus identities.
 %  which are again all trivially checked except for the case \((\Lie_X,\omega)\), which follows from commutation rule of $\ddR$ and $\Lie_X$ (equation \eqref{eq:CartanLievsD}).
 Finally, the equation \eqref{eq:Linfty-morphism-DGLAs-arity-3}  is 
    \begin{align}
        \Phi_2&(\lbrace v_1,v_2\rbrace_{\gSmall_r},v_3) 
        -(-1)^{|v_2||v_3|} \Phi_2(\lbrace v_1,v_3\rbrace_{\gSmall_r},v_2)
        +(-1)^{|v_1|(|v_2|+|v_3|)} \Phi_2(\lbrace v_2,v_3\rbrace_{\gSmall_r},v_1)
        \\
        =& (-1)^{|v_1|}\lbrace \Phi_1(v_1),\Phi_2(v_2,v_3)\rbrace_{\gSmall_{r+1}}-
        (-1)^{|v_2|+|v_1||v_2|} \lbrace \Phi_1(v_2),\Phi_2(v_1,v_3)\rbrace_{\gSmall_{r+1}} +
        (-1)^{|v_3|(|v_1|+|v_2|)} \lbrace \Phi_1(v_3),\Phi_2(v_1,v_2)\rbrace_{\gSmall_{r+1}}~.
        \nonumber
    \end{align}
    for any $v_1,v_2\,v_3\in \gSmall_r$. Specializing the entries $v_i$ to be in  $\X$, in $\X[1]$ or in $\Omega$, this equation is immediately seen to be satisfied in all cases but $(v_1,v_2,v_3)=  (X_{[1]},Y_{[1]},\omega)$ and  $(v_1,v_2,v_3)=  (X,Y_{[1]},\omega)$, that require a little computation.
 In the first case, we find the identity 
    \begin{align*}
        \cancel{\Phi_2(0,\omega)} 
        &-(-1)^{|\omega|} \Phi_2(\iota_{X}\omega,Y_{[1]}) 
        - (-1)^{|\omega|} \Phi_2(\iota_{Y}\omega,X_{[1]}) 
        \\
        =&~ 
        - \lbrace X_{[1]},\iota_{Y} \omega\rbrace_{\gSmall_{r+1}} 
        - \lbrace Y_{[1]},\iota_{X} \omega\rbrace_{\gSmall_{r+1}} 
        \cancel{+ \lbrace \ddR \omega,0\rbrace_{\gSmall_{r+1}}}~,
    \end{align*}
    that is the identity
    \begin{displaymath}
        -(-1)^{2|\omega|+1} \iota_Y \iota_{X} \omega 
        -(-1)^{2|\omega|+1} \iota_{X} \iota_{Y} \omega = 
        -\iota_{X} \iota_{Y} \omega 
        - \iota_{Y} \iota_{X} \omega~,
    \end{displaymath}
    which is manifestly satisfied.
    %and it holds true due to Cartan's equation \eqref{eq:CartanIotasquared}.
    %The forth case \((\Lie_{X_1},\iota_{X_2},\omega_3)\) reads
     In the second case, we find the identity 
    \begin{align*}
        \Phi_2({[X,Y]_\X}_{[1]},\omega) 
        &-(-1)^{|\omega|} \Phi_2(\Lie_{X}\omega,Y_{[1]})
        + \cancel{\Phi_2(\iota_{Y}\omega,X)}
        \\
        =&~
        \lbrace {X},\iota_{Y} \omega\rbrace_{\gSmall_{r+1}} 
        - \cancel{\lbrace {Y}_{[1]},0\rbrace_{\gSmall_{r+1}}} 
        + \cancel{\lbrace \ddR\omega,0\rbrace_{\gSmall_{r+1}}}~,
    \end{align*}
 that is the identity
    \begin{displaymath}
        \iota_{[X,Y]_\X} \omega 
        + \iota_{Y} \Lie_{X} \omega - 
        \Lie_{X} \iota_{Y} \omega =0~,
    \end{displaymath}
    which is nothing but  Cartan's identity \eqref{eq:CartanLievsIota}.
    
Finally, $\Phi$ restricts to an \(L_\infty\)-morphism \(\Phi\vert_{{\ge 0}} : \g_{r}^{\ge 0} \to \g_{r+1}^{\ge 0}\) %given by $(\Phi_1 \vert_{\gSmall_{r}^{\ge 0}}, 0,\dots)$
since, for any $v= X + \omega \in \gSmall_{r}^{\ge 0}$, i.e., with $\deg(\omega)\geq r$, one has
\begin{displaymath}
    \Phi_1(v) = \ddR \omega + X \in \gSmall_{r+1}^{\ge 0}~,
\end{displaymath}
and
\begin{displaymath}
    \Phi_2(v,-) = 0~,
\end{displaymath}
when the second entry is in $\gSmall_{r+1}^{\ge 0}$.
   % In other terms, the above diagram commutes on the nose.
\end{proof}
Summing up, we have proven \autoref{bigthm:morphism-twistedCourant-untwistedCourant} from the Introduction in the special case $\sigma=0$. We address the general case in the following subsection.

%Wrapping up, we have proven the following.
%\begin{theorem}
%    There exists a canonical \(L_\infty\)-morphism
 %   \[
  %  \Courant{\,r-1} \longrightarrow \Courant{r}~.
   % \]
%\end{theorem}

%-------------------------------------------------+
\subsection{The canonical morphism between \texorpdfstring{$\Courant[\sigma]{\,r-1}$ and $\Courant[]{r}$}{twisted and untwisted higher Courant L-infinity-algebras}}\label{ssec:theoremB}
%-------------------------------------------------+
Recall that any closed $r+1$-form \(\sigma\in \Omega^{r+1}_{\mathrm{cl}}\) defines a Maurer-Cartan element of the DGLA \(\g_{r}\) hence it can be used to define a twisted DGLA \(\g_{r,\sigma}\).
A Maurer-Cartan element can also be used to twist a given morphism of DGLAs, as recalled in the following classical lemma.
\begin{lemma}[Twisting \(L_\infty\)-morphisms]
    Let \(\Phi:\g\to \h\) be an \(L_\infty\)-morphism between two DGLAs \((\g,\d_\g,\lbrace-,-\rbrace_\g)\) and \((\h,\d_\h,\lbrace-,-\rbrace_\h)\). For any Maurer-Cartan element \(\mu\in \g\), set
    \[
    \Phi(\mu):=\sum_{k\ge 1}\frac{1}{k!}\,\Phi_k\Big(\underbrace{\mu,\dots,\mu}_{k}\Big).
    \]
    Then \(\Phi(\mu)\) is a Maurer-Cartan element of \(\h\) and the maps
    \[
    \big(\Phi_\mu\big)_k(-)=\sum_{j\ge 0}\frac{1}{j!}\,\Phi_{j+k}\Big(\underbrace{\mu,\dots,\mu}_{j},-\Big),
    \]
    are the $k$-ary components of an \(L_\infty\)-morphism \(\Phi_\mu:\g_\mu\to \h_{\Phi(\mu)}\) between the twisted DGLAs \((\g,\d_\g+\lbrace \mu,-\rbrace_\g,\lbrace-,-\rbrace_\g)\) and \((\h,\d_\h+\lbrace \Phi(\mu),-\rbrace_\h,\lbrace-,-\rbrace_\h)\).
\end{lemma}
Specializing this to the case we are interested in, for any given $\sigma \in \Omega^{r+1}_{\mathrm{cl}}$, one obtains an $L_\infty$-morphism 
\[
\Phi_\sigma:\gSmall_{r,\sigma}\to \gSmall_{r+1,\Phi(\sigma)}
\]
 by twisting the $L_\infty$-morphism \(\Phi:\g_{r}\to \g_{r+1}\) introduced in \autoref{lem:Linfty-morphism-smallgr-to-smallgr+1}.
 One has
 \[
 \Phi(\sigma)=\Phi_1(\sigma)+\frac{1}{2}\Phi_2(\sigma,\sigma)=\ddR \sigma 0,
 \]
 hence $\gSmall_{r+1,\Phi(\sigma)}=\gSmall_{r+1}$. The $k$-ary components of $\Phi_\sigma$ are given by:
\begin{align*}
    \big(\Phi_\sigma\big)_1(v) 
    &= \Phi_1(v) + \Phi_2(\sigma,v)~, \\
    \big(\Phi_\sigma\big)_2(v_1,v_2) 
    &= \Phi_2(v_1,v_2), \\
    \big(\Phi_\sigma\big)_k(v_1,\dots,v_k) 
    &= 0~, \quad k\ge 3~,
\end{align*}
so, just like $\Phi$, the $L_\infty$-morphism  $\Phi_\sigma$ restricts to an $L_\infty$-morphism  from $\gSmall_{r,\sigma}^{\ge 0}$ to $\gSmall_{r+1}^{\geq 0}$. In other words, we have a strictly commutative diagram of $L_\infty$-morphisms 
\[
\begin{tikzcd}
\g_{r,\sigma}^{\ge 0} \ar[r,hook] \ar[d, "\Phi\vert_{\ge 0}"'] &
\g_{r,\sigma} \ar[d, "\Phi_\sigma"] \\
\g_{r+1}^{\ge 0} \ar[r,hook] &
\g_{r+1} .
\end{tikzcd}
\]
By the universal property of homotopy fibers, we therefore get \autoref{bigthm:morphism-twistedCourant-untwistedCourant} from the Introduction. Notice that, by construction, the linear part of the $L_\infty$-morphism $\Courant[\sigma]{r-1} \to \Courant[]{r}$ obtained this way is (homotopy equivalent) to the morphism of cochain complexes \eqref{eq:courant-linear}.

%Via the universal property of homotopy fibers, one can conclude the following.
%%
%\begin{theorem}[Morphism from twisted  to untwisted Courant $L_\infty$-algebras]
%    \label{thm:morphism-Twisted-UntwistedCourant}
%    There is an \(L_\infty\)-commutative diagram
%    \begin{displaymath}
%       \begin{tikzcd}
%            \Courant[\sigma]{r-1} \ar[d]\ar[r] & 
%            \gSmall_{r,\sigma}^{\ge 0} \ar[r,hook]\ar[d,"\Phi\vert_{\ge 0}"] &
%            \gSmall_{r,\sigma} \ar[d,"\Phi"]
%            \\
%            \Courant[]{r} \ar[r] &
%            \g_{r{+}1}^{\ge 0} \ar[r,hook] %\ar[ur,Rightarrow] 
%            &
%            \g_{r{+}1},
%        \end{tikzcd}
%    \end{displaymath}   
%where \(\Phi\) is a $2$-term  $L_\infty$-morphism, \(\Courant[\sigma]{r-1}\cong \hoFib(\g_{r,\sigma}^{\ge 0}\hookrightarrow \g_{r,\sigma})\) and \(\Courant{r}\cong \hoFib(\g_{r+1}^{\ge 0}\hookrightarrow \g_{r+1})\).
%\end{theorem}

%=================================================#
\section{A morphism between the multisymplectic \texorpdfstring{$L_\infty$}{L-infinity}-algebra and the twisted higher Courant \texorpdfstring{$L_\infty$}{L-infinity}-algebra}
\label{sec:RogerstoCourant}
%=================================================#

%Let $M$ be a smooth manifold and $\sigma \in \Omega^{r+2}_{\textrm{cl}}(M)$ be a pre-$(r+1)$-plectic form with $r\geq0$ integer.
Consider now the Lie algebra $\X_{\ham,\sigma}(M)$  of Hamiltonian vecor fields for the pre-$(r+1)$-plectic manifold $(M,\sigma)$. Denote by $\Lie\colon  \X_{\ham,\sigma}(M) \to \g_{r+1,\sigma}^{\ge 0}$ the canonical embedding induced by the inclusion $\X_{\ham,\sigma}(M)\subseteq \X(M)$, and by $\iota: \X_{\ham,\sigma}(M) \to \g_{r+1,\sigma}[-1]$ the map sending a Hamiltonian vector field $X$ to the degree $-1$ element $X_{[1]}$ in $\g_{r+1,\sigma}$. Since $\mathcal{L}$ is a DGLA morphis, 
we can consider the following diagram of $L_\infty$-morphisms:
\begin{equation}\label{eq:diag-Rogers-DGLA}
    \begin{tikzcd}
        \X_{\ham,\sigma}(M) \ar[r,"\iota_{\infty}\sigma"] \ar[d,"\Lie"'] &
        \Trunc\Big(\Omega(M)[r+1]\Big) \ar[d,hook]
        \\
        \g_{r+1,\sigma}^{\ge 0} \ar[r,hook] &%\ar[ur,Rightarrow] &
        \g_{r+1,\sigma}
    \end{tikzcd}
\end{equation}
where $\X_{\ham,\sigma}(M)$ is the Lie algebra of Hamiltonian vector fields on $M$ (see \autoref{def:hamiltonian-pairs}), $\Trunc(\Omega(M)[r+1])$ is the abelian DGLA (i.e., the chain complex) given by the canonical truncation of the shifted de Rham complex 
   $$
    \Trunc\Big(\Omega[r+1]\Big) := 
    \left(
        \begin{tikzcd}[column sep = normal]
            \text{\tiny (deg $-r-1$)} & \text{\tiny (deg $-r$)} &  \cdots & \text{\tiny (deg $-1$)} & \text{\tiny (deg $0$)}\\[-2em]
            \Omega^0 \ar[r,"\d"] & \Omega^1 \ar[r,"\d"]  &
            \cdots \ar[r,"\d"] & \Omega^{r} \ar[r,"{\d}"] & \d\Omega^{r}
        \end{tikzcd}
    \right)~,
    $$
and $\iota_{\infty}\sigma$ is the $L_\infty$-morphism defined by contractions against $\sigma$ as in \eqref{eq:iota-infinity-sigma}.
The diagram \eqref{eq:diag-Rogers-DGLA} does not strictly commute, but it does so up to homotopy. To see this, we use the known fact, recalled in the Appendix as \autoref{prop:Linfty-homotopy-gauge-equivalence} that $L_\infty$-morphisms between two DGLAs $\mathfrak{g}$ and $\mathfrak{h}$ are the Maurer-Cartan elements of a Chevalley--Eilenberg-type DGLA $\CE(\mathfrak{g},\mathfrak{h})$ and that two $L_\infty$-morphisms between $\mathfrak{g}$ and $\mathfrak{h}$ are homotopy equivalent if and only if they are gauge equivalent as Maurer-Cartan elements in $\CE(\mathfrak{g},\mathfrak{h})$.
\begin{lemma}
    The bottom-right composition $\Lie$ and the right-bottom composition $\iota_\infty\sigma$ in the diagram \eqref{eq:diag-Rogers-DGLA}, seen as Maurer-Cartan elements in $\CE(\X_{\ham,\sigma}(M),\g_{r+1,\sigma})$, are related by the gauge transformation generated by the degree $0$ element $\iota: \X_{\ham,\sigma}(M) \to \g_{r+1,\sigma}[-1]$, i.e.
    \begin{displaymath}
        \iota_\infty\sigma = e^{\iota} \star \Lie~.
    \end{displaymath}
\end{lemma}
\begin{proof}
    Denote for brevity the Chevalley-Eilenberg complex $\CE(\X_{\ham,\sigma}(M),\g_{r+1,\sigma})$ by $\mathfrak{ce}$.
    Notice that $\Lie$ and $\iota_\infty\sigma$ are degree $1$ elements in $\mathfrak{ce}$ which satisfy the  Maurer-Cartan equation. 
    On the other hand, the map $\iota$ is a degree $0$ element in $\mathfrak{ce}$.
    %the same complex since maps vector fields to contraction operators, which are degree $-1$ elements in $\g_{r,\sigma}$.
 The gauge action of $\iota$ on $\Lie\in \MC(\mathfrak{ce})$, see equations \eqref{eq:MC-Gauge-Transformation} and \eqref{eq:MC-Gauge-Transformation-CE},
  is given by the formula
 \begin{displaymath}
     e^{\iota} \star \Lie = \Lie + \sum_{n=1}^\infty \frac{[\iota, \blank]_{\mathfrak{ce}}^{n-1}}{n!}([\iota, \Lie]_{\mathfrak{ce}} - \d_{\mathfrak{ce}}(\iota))~.
 \end{displaymath}
 According to \autoref{def:CE-DGLA-Linfty-morphisms}, the last term reads as follows:
    \begin{align*}
        \d_{\mathfrak{ce}}(\iota)
        =&~ 
        \iota \circ [\blank,\blank]_{\X} + \cancel{\iota \ca \d_{\X}}
        +\d_{\g_{r+1,\sigma}} \ca \iota
        \\
        =&~
        \iota \circ [\blank,\blank]_{\X} 
        +\d_{\g_{r+1,\sigma}} \circ \iota        
        ~,
    \end{align*}
    where the crossed term vanishes since $\X$ is concentrated in degree $0$ and the sign of the last term comes from the fact that the arity and degree of $\iota$ are both equal to $1$.
    Unpacking the definition of the $\sigma$-twisted differential $\d_{\g_{r,\sigma}}$, one finds that
    \begin{align*}
        \d_{\g_{r+1,\sigma}} \circ \iota &= \cancel{(\ddR\oplus 0) \circ \iota} + \Lie + \lbrace\sigma, \iota\rbrace_{\g_{r+1}} ~,
        \\
        &= \Lie + \iota_{\blank}\sigma~,
    \end{align*}
    where the cancellation occurs because the map $\iota$ lands in $\mathfrak{C}^{-1}$. %$\Cartan^{-1}$.
    Moreover, according to equation \eqref{eq:bracket-Hom-DGLAs}, one finds that
    \begin{align*}
        [\iota, \Lie]_{\mathfrak{ce}} &= 
        -
         \lbrace\blank,\blank\rbrace_{\g_{r+1}} \circ (\iota \otimes \Lie) \circ P_{(1,1)}\\
        &={-}  \lbrace\iota, \Lie\rbrace_{\g_{r+1}}\circ P_{(1,1)}
        \\
        &=  2 \iota\circ {[\blank,\blank]_{\X}}~.
    \end{align*}
    From these, one obtains
    \begin{displaymath}
        e^{\iota} \star \Lie = \Lie + \sum_{n=1}^\infty \frac{[\iota, \blank]_{\mathfrak{ce}}^{n-1}}{n!} \left( \left[\iota,\frac{\Lie}{2}\right]_{\mathfrak{ce}}-\Lie - \iota_{\blank}\sigma \right)~.
    \end{displaymath}
    The component in arity $1$ of the above expression reads as
    \begin{align*}
        \left( e^{\iota}\star \Lie\right)_1
        =&~
        \Lie - \Lie - \iota_{\blank}\sigma =
        - \iota_{\blank}\sigma =
        \left(\iota_{\infty}\sigma\right)_1~.
    \end{align*}
    In arity $k\geq 2$, one has
    \begin{align*}
        \left( e^{\iota}\star \Lie\right)_k
        =&~
        \frac{1}{(k-1)!} \left[\iota, \blank\right]_{\mathfrak{ce}}^{k-1} \circ \left(\frac{\Lie}{2}-\frac{\Lie}{k}\right)
        + \frac{1}{k!}\left[ \iota,\blank\right]_{\mathfrak{ce}}^{k-1}\circ \left(\iota_{\infty}\sigma\right)_1
        ~.
        \\
    \end{align*}
    The first term manifestly vanishes when $k=2$. When $k\geq 3$, one can show by induction that it vanishes as well since
        \begin{align*}
            \frac{k-2}{k}\left[\iota,\left[\iota,\frac{\Lie}{2}\right]_{\mathfrak{ce}}\right]_{\mathfrak{ce}} =&~
             \frac{2-k}{k}
             \lbrace\blank,\blank\rbrace_{\g_{r+1}} \circ \left(\iota \otimes \left[\iota,\frac{\Lie}{2}\right]_{\mathfrak{ce}}\right) \circ P_{(1,1)}\\
             =&~
             \frac{2-k}{k}
             \lbrace\blank,\blank\rbrace_{\g_{r+1}} \left(\iota \otimes \left(\iota\circ {[\blank,\blank]_{\X}} \circ P_{(1,1)}\right)\right)
             \circ P_{1,2}
             \\
             =&~
           \frac{2-k}{k}
             \lbrace\iota(\blank),\iota\circ [\blank,\blank]_{\X}\rbrace_{\g_{r+1}}\circ P_{1,1,1}
                \\
                =&~
                0~,
        \end{align*}
        where the last cancellation occurs because the bracket $\lbrace\cdot,\cdot\rbrace_{\g_{r+1}}$ between two elements in $\X[1]\subseteq \g_{r+1}$
        %contraction operators 
        is zero.
        \\
        On the other hand, the last term coincides with $-(\iota_\infty \sigma)_k$ since %, as can be easily checked for $k=2$, 
        one has
        \begin{align*}
            \left[ \iota,\left(\iota_\infty \sigma\right)_k
            \right]_{\mathfrak{ce}} 
            =&~
            %{\color{purple}(-1)^{1 (k - k) }}
             \lbrace \iota, \iota_\blank\sigma\rbrace_{\g_{r+1}} \circ P_{(1,k)}\\
            =&~
            (k+1) \left(\iota_\infty\sigma\right)_{k+1}~,
        \end{align*}
        and therefore, by induction, one gets
        \begin{align*}
            \left[ \iota, \blank \right] ^{k-1} \, \left(\iota_\infty\sigma\right)_1 
            &=~
            k !\left(\iota_\infty \sigma\right)_{k}~.
        \end{align*}
\end{proof}
Summing up, we have proven the following.
\begin{proposition}\label{prop:gauge-for-ham}
 There is a homotopy commutative diagram of $L_\infty$-algebras
    \begin{displaymath}
       \begin{tikzcd}
           \X_{\ham,\sigma} \ar[r,"\iota_{\infty}\sigma"] \ar[d,"\Lie"'] &[2em]
            \Trunc\Big(\Omega(M)[r]\Big)
            %(\Omega^0\to\cdots\to \Omega^{r-1}\to \d\Omega^{r-1})
             \ar[d,hook]
            \\
            \g_{r,\sigma}^{\ge 0} \ar[r,hook] \ar[ur,Rightarrow, "e^{t\iota}"] &
            \g_{r,\sigma},
        \end{tikzcd}
    \end{displaymath}
    where $e^{t\iota}$ is the homotopy associated with the gauge transformation $e^\iota$.
\end{proposition} 
From the universal property of homotopy fibers, one gets \autoref{bigthm:morphism-Rogers-Courant} from the Introduction. Summing up, we have the following Theorem, summarizing Theorems \ref{bigthm:smaller-model}, \ref{bigthm:morphism-twistedCourant-untwistedCourant}, and \ref{bigthm:morphism-Rogers-Courant}.
\begin{theorem}[Theorems A,B,C]
    \label{thm:ABC}
    There is a distinguished homotopy commutative diagram of $L_\infty$-algebras
\begin{equation}
    \label{eq:final-diagram}
       \begin{tikzcd}
            \Rogers[\sigma]{r-1} \ar[d]\ar[r] &[3em] \X_{\ham,\sigma} \ar[r,"\iota_{\infty}\sigma"] \ar[d,"\Lie"'] &[4em]
            \Trunc\Big(\Omega(M)[r]\Big)
             \ar[d,hook]
            \\
            \Courant[\sigma]{r-1} \ar[r] \ar[d]&
            \g_{r,\sigma}^{\ge 0}\ar[d, "\Phi\vert_{\ge 0}"']  \ar[r,hook] \ar[ur,Rightarrow, "e^{t\iota}"] &
            \g_{r,\sigma}\ar[d, "\Phi_\sigma"]
            \\
            \Courant[\sigma]{r} \ar[r] &
\g_{r+1}^{\ge 0} \ar[r,hook] &
\g_{r+1}         \end{tikzcd}
    \end{equation}
    where the universal property of homotopy fibers induces the leftmost vertical arrows.
\end{theorem}

\begin{remark}
An explicit expression for the\footnote{Or, more precisely, an explicit representative of its homotopy equivalence class.} $L_\infty$-morphism $\Rogers[\sigma]{r-1}\to \Courant[\sigma]{r-1}$ constructed above is explicitly given in \cite{Miti2024}.
\end{remark}

\begin{remark} When $\sigma=0$ the equation $e^{\iota} \star \Lie=\iota_\infty\sigma$ reduces to the equation $e^{\iota} \star \Lie=0$ considered in \cite{Fiorenza2009} in the context of Cartan homotopies (\cite{Fiorenza2012,Iacono2025}).
\end{remark}

%\antonio{per Domenico: puoi aggiungere tu un remark di connessione alle \emph{Cartan homotopies} menzionate nei paper con Iacono e Martinengo~\cite{Fiorenza2009,Fiorenza2012,Iacono2025}?}

\begin{remark}[Algebraic generalizations and explicit formulas]
  %  \ngt{
In the spirit of \autoref{rem:outlook-algebraic-generalizations}, one may aim at a purely algebraic extension of \autoref{thm:ABC}.
In such a setting, the existence of suitable $L_\infty$-morphisms between homotopy fibers defined entirely in algebraic terms could be established by arguments analogous to the homotopical techniques employed above.
%
%A natural subsequent question concerns the derivation of explicit formulas for these morphisms, in the same vein as the constructions carried out in \cite{Miti2024} for the case of \autoref{thm:morphism-Rogers-Courant}.
%%
%Recent work of Getzler \cite{Getzler2025} provides explicit expressions for the relevant pullback maps, as well as a related groupoid-level formulation.
%According to his approach, explicit formulas for the $L_\infty$-morphisms under consideration can be obtained via cubical-set techniques.
%Related ideas also suggest the possible use of Bäcklund-type transformations ...
 %   }
\end{remark}

%\antonio{per Domenico: puoi aiutarmi a completare/aggiustare il precedente remark? In particolare relativamente al discorso di Backlund.
%\\
%Discussions with Getzler:
% \\- A recent paper by Getzler (11/3/25, said to be ``today'', possibly arXiv:2408.11157) contains an explicit expression for the pullback. Another of his articles has a ``groupoid version''.
% \\- He claims that explicit formulas for our $L_\infty$-morphisms can be obtained from the approach via cubical sets.
% \\- Domenico suggests that we could use the Bäcklund map. According to Getzler, this is related to the holonomy of flat connections (an unpublished result of Kapranov).
% \\- A paper by Bandiera on semiabelian $L_\infty$-algebras yields higher groupoids (strict $n$-groupoids).
%}

%+------------------------------------------------+
\subsection{Homotopy comomentum maps and the pentagonal diagram}
In \cite{Miti2024}, the authors consider a pentagonal diagram inspired by the geometric prequantization framework for integral symplectic forms.  
The commutativity of this pentagonal diagram admits a geometric justification whenever a suitable factorization through a so-called prequantization map is available.
In the literature, several notions of higher prequantization have been proposed, for instance, \cite[\S 5.1]{Krepski2022}, \cite{Djounvouna2024}, \cite[Thm. 3.5]{Sevestre2021}, \cite[Proposition 5.10]{Fiorenza2014}, and \cite{Fiorenza2016}.
However, none of these constructions appears to retain all the desirable properties of classical prequantization simultaneously.
In this section, we revisit the pentagonal diagram from \cite{Miti2024}, providing a homotopical justification for its commutativity.

Let $(M,\sigma)$ be a pre-$(r)$-plectic manifold, and let $\sigma' = \sigma + \ddR \beta$ be another pre-$(r)$-plectic form in the same de~Rham cohomology class, for some $\beta \in \Omega^{r}(M)$.
One can form the following open diagram in the category of $L_\infty$-algebras:
    \begin{displaymath}
       \begin{tikzcd}
            \Rogers[\sigma]{r-1}\ar[r] &
            \Courant[\sigma]{r-1} \ar[d,leftrightarrow,sloped,"\sim"]
            \\
            \Rogers[\sigma']{r-1}  \ar[r] &
            \Courant[\sigma']{r-1}
        \end{tikzcd}
        ~.
    \end{displaymath}
The rightmost equivalence takes the name of \emph{gauge transformation} (also $\beta$-transformation) in the literature on higher symplectic geometry. An explicit construction can be found in \cite[\S 5.1]{Miti2024}.
Such an equivalence can also be derived by the universal property of homotopy fibers, noticing that the two DGLAs $\gSmall_{r,\sigma}$ and $\gSmall_{r,\sigma'}$ are isomorphic, as shown below.
\begin{lemma}[Gauge transformation between twisted higher Courant DGLAs]
    \label{lem:gauge-transf-twisted-Courant-DGLAs}
    Let \(\sigma,\sigma' \in \Omega^{r+1}_{\mathrm{cl}}(M)\) be two closed forms such that \(\sigma'=\sigma + \ddR \beta\) for a given \(\beta \in \Omega^{r}(M)\).
    Then, the map
    \begin{displaymath}
        \morphism{\chi_{\beta}}{\g_{r,\sigma}}{\g_{r,\sigma'}}
        {\omega + X_{[1]} + Y}
        {\omega + X_{[1]} + \iota_X \beta + Y + \Lie_Y \beta},
    \end{displaymath}
    defines an isomorphism of DGLAs $\g_{r,\sigma} \cong \g_{r,\sigma'}$.
    This induces the commutative square of DGLAs:
    \begin{displaymath}
       \begin{tikzcd}
            \g_{r,\sigma}^{\ge 0} \ar[r,hook] \ar[d,"\chi_{\beta}\vert_{\ge 0}"',"\wr"] 
            &[3em]
            \g_{r,\sigma} \ar[d,"\chi_{\beta}","\wr"']
            \\
            \g_{r,\sigma'}^{\ge 0} \ar[r,hook] &
            \g_{r,\sigma'},
        \end{tikzcd}
    \end{displaymath}
    and so an equivalence $\Courant[\sigma]{r-1} \cong \Courant[\sigma']{r-1}$ by the universal property of homotopy fibers.
\end{lemma}
\begin{proof}
It is a general fact about DGLAs $\mathfrak{g}_{\lambda}$ and $\mathfrak{g}_{\lambda'}$ obtained by twisting the differential $\d$ of a DGLA $\mathfrak{g}$ by gauge equivalent Maurer-Cartan elements $\lambda$ and $\lambda'$: if $\lambda'=e^\gamma\star \lambda$ with a nilpotent $\gamma$, then  
\[
e^{\mathrm{ad}_\gamma} \colon \mathfrak{g}_{\lambda}\to \mathfrak{g}_{\lambda'}
\]
is a DGLA isomorphism (see e.g. \cite[\S 6.3]{Manetti2022}). 
Indeed, since $[\gamma,\blank]$ is a derivation of the bracket of $\mathfrak{g}$, the exponential $e^{[\gamma,\blank]}$ is an automorphism of the graded Lie algebra $\mathfrak{g}$. So it is an isomorphism of graded Lie algebras from $\mathfrak{g}_\lambda$ to $\mathfrak{g}_{\lambda'}$. To see its action on the differentials, we compute
\[
e^{\mathrm{ad}_\gamma}\circ \d_\lambda \circ e^{-\mathrm{ad}_\gamma}=e^{\mathrm{ad}_\gamma}\circ (\d+[\lambda,-]) \circ e^{-\mathrm{ad}_\gamma}=d+[e^\gamma\star\lambda,-]=\d+[\lambda',\blank]=\d_\lambda',
\]
and so
\[
e^{\mathrm{ad}_\gamma}(\d_\lambda x)= \d_{\lambda'}(e^{\mathrm{ad}_\gamma} x).
\]
In our case 
\[
\sigma'=\sigma+\ddR\beta=e^{-\beta}\star \sigma,
\]
and
$\mathrm{ad}_{-\beta}$ maps $\mathfrak{g}_r$ into the subcomplex $\Omega[r]$, so that $\mathrm{ad}_{-\beta}^2=0$. In particular $\mathrm{ad}_{-\beta}$ is nilpotent. Therefore $\chi_\beta:=e^{\mathrm{ad}_{-\beta}}\colon \mathfrak{g}_{r,\sigma}\to \mathfrak{g}_{r,\sigma'}$ is a DGLA isomorphism. Since $\mathrm{ad}_{-\beta}^2=0$,
the isomorphism $\chi_\beta:=e^{\mathrm{ad}_{-\beta}}$ reduces to
\[
\omega + X_{[1]} + Y\mapsto \omega + X_{[1]} + Y
-\{\beta,\omega + X_{[1]} + Y\}_{\mathfrak{g}_r}=\omega + X_{[1]} 
+ Y
+\iota_X\beta + \Lie_Y\beta,
\]
where we used that $|\beta|=\deg(\beta)-r=0$.
Being an isomorphism of DGLAs, $\chi_\beta$ clearly preserves the non-negative parts.
\end{proof}

    The bottom line of \cite[Thm.~5.3]{Miti2024} is that the open diagram above can be completed to a commutative pentagon, provided there exists a Lie algebra action preserving the pre-\((r)\)-plectic forms $\sigma$ and $\sigma'$, and admitting a so-called \emph{homotopy comomentum map}.
    \begin{definition}[Homotopy comomentum maps {\cite{Callies2016}}]
        Let $(M,\sigma)$ be a pre-$(r)$-plectic manifold and $\mathfrak{K}$ be a Lie algebra acting on $M$ via a Lie algebra morphism $\rho:\mathfrak{K} \to \X_{\ham,\sigma}(M)$.
        A \emph{homotopy comomentum map} is a morphism of $L_\infty$-algebras $\mathcal{h}:\mathfrak{K}\to \Rogers[\sigma]{r-1}$ lifting the infinitesimal action, i.e.,  the datum of a homotopy commutative diagram of \(L_\infty\)-algebras of the form
        \begin{displaymath}
              \begin{tikzcd}[]
                  \mathfrak{K} \ar[rd,"\rho"',shift right = 1.5ex,""{name=A}] \ar[r,"\mathcal{h}"] &[5em] \Rogers[\sigma]{r-1} \ar[d]
                 \\
                &[1em]
                 \X_{\ham,\sigma}(M)
                 \arrow[Rightarrow,"\eta"', from=A, to=1-2]
                \end{tikzcd}
        \end{displaymath}
        
        \end{definition}
    \begin{remark}[Homotopy comomentum maps are homotopies]
        \label{rem:homotopy-comomentum-maps-are-homotopies}
    By the universal property of homotopy fibers, a homotopy comomentum map is equivalently given by an $L_\infty$-homotopy filling the  commutative square
    \begin{displaymath}
        \begin{tikzcd}
            \mathfrak{K} \ar[r] \ar[d,"\rho"'] &[5em] 0 \ar[d]
            \\[1em]
            \X_{\ham,\sigma}(M) \ar[r,"\iota_{\infty}\sigma"'] \ar[ur,Rightarrow,sloped,"\mathcal{h}"] & 
            \Trunc\Big(\Omega(M)[r]\Big)~.
            %(\Omega^0\to\cdots\to \Omega^{r-1}\to \d\Omega^{r-1})~.
        \end{tikzcd}
    \end{displaymath}
    %where $\Trunc(\Omega(M)[r])$ is the abelian DGLA $(\Omega^0\to\cdots\to \Omega^{r-1}\to \d\Omega^{r-1})$.
    As such, looking at $\iota_{\infty}\sigma \circ \rho$ and the zero map as Maurer-Cartan elements in $\mathfrak{ce}=\CE(\mathfrak{K},\Trunc(\Omega(M)[r]))$, a homotopy comomentum map is equivalently given by a gauge transformation between them.
    Namely, it can be seen as a degree zero element $\mathcal{h} \in \mathfrak{ce}^0$, i.e., a collection of maps
    \begin{displaymath}
        \lbrace h_k:\wedge^k \mathfrak{K} \to \Omega^{r-k}(M) ~|~ 1\leq k \leq r \rbrace~,
    \end{displaymath}
    satisfying the equation
    \begin{displaymath}
        e^{\mathcal{h}} \star \left(\iota_\infty \sigma \circ \rho\right)
        = 0~.
    \end{displaymath}
    Being $\Trunc(\Omega(M)[r])$ an abelian DGLA, according to equations \eqref{eq:MC-Gauge-Transformation-CE} and \eqref{eq:bracket-Hom-DGLAs}, the gauge action simplifies to
    \begin{displaymath}
        e^{\mathcal{h}} \star \left(\iota_\infty \sigma \circ \rho\right)
        = \iota_\infty \sigma \circ \rho - \d_{\mathfrak{ce}}(\mathcal{h})~,
    \end{displaymath}
    where the differential $\d_{\mathfrak{ce}}$ reads 
    \begin{displaymath}
        \d_{\mathfrak{ce}}(\mathcal{h}) = \cancel{\mathcal{h} \ca \d_{\mathfrak{K}}} + \ddR \ca \mathcal{h} + \mathcal{h} \ca [\blank,\blank]_{\mathfrak{K}}
        ~,
    \end{displaymath}
    since $\mathfrak{K}$ is concentrated in degree $0$.
    Therefore, the homotopy comomentum map condition is equivalent to 
    \begin{displaymath}
        \iota_\infty \sigma \circ \rho = \ddR \ca \mathcal{h} + \mathcal{h} \ca [\blank,\blank]_{\mathfrak{K}}~
    \end{displaymath}
    which, when read on componets of fixed arity $k$, gives exactly the explicit conditions listed in \cite{Callies2016,Miti2021,Miti2024}.
    \end{remark}
    Using these conditions, an explicit computation in  \cite{Miti2024} proves the commutativity of the following diagram of \(L_\infty\)-algebras:
    \begin{equation}
        \label{eq:pentagonal-diagram}
       \begin{tikzcd}
            &[1em]
            \Rogers[\sigma]{r-1}\ar[r] &
            \Courant[\sigma]{r-1} \ar[dd,leftrightarrow,sloped,"\sim"]
            \\[-.75em]
            \mathfrak{K} \ar[ur,"\mathcal{h}"] \ar[dr,"\mathcal{h}'"'] &
            &
            \\[-.75em]
            &
            \Rogers[\sigma']{r-1}  \ar[r] &
            \Courant[\sigma']{r-1}
        \end{tikzcd}
        ~,
    \end{equation}
    where $\mathfrak{K}$ is a Lie algebra acting on $M$ via Hamiltonian vector fields preserving both $\sigma$ and $\sigma'$, and admitting homotopy comomentum maps $\mathcal{h}$ and $\mathcal{h}'$ with respect to $\sigma$ and $\sigma'$, respectively.

    Our goal is to establish the naturality of diagram \eqref{eq:pentagonal-diagram}, namely to show that its commutativity is a formal consequence of the universal property of homotopy fibers.
To begin with, observe that we have a natural equivalence $\mathfrak{K} \cong \hoFib(\mathfrak{K} \to 0)$. %, since it is the homopy pullback  of  $\mathfrak{K} \to 0$ along  the canonical isomorphism $0 \cong 0$.
By the universal property of homotopy fibers, the homotopy commutativity of diagram \eqref{eq:pentagonal-diagram} is therefore equivalent to the homotopy commutativity of the lateral and top faces of the following prism-shaped diagram.

    \begin{equation}
        \label{eq:prism-diagram}
        \begin{tikzcd}[row sep=.5em,
  column sep=.5em,
  arrows={line width=0.35pt}]
	        & \X_{\ham,\sigma} 
            &&& \gSmall_{r+1,\sigma}^{\ge 0}
            \\
            \mathfrak{K} 
            \\
            && \X_{\ham,\sigma} 
            &&& \gSmall_{r+1,\sigma'}^{\ge 0}
            \\
	        \\
	        & \Trunc\Big(\Omega(M)[r]\Big)
            &&& \gSmall_{r+1,\sigma} 
            \\
            0 
            \\
            && \Trunc\Big(\Omega(M)[r]\Big)
            &&& \gSmall_{r+1,\sigma'}
	\arrow[from=1-2, to=1-5]
	\arrow[from=1-2, to=5-2]
	\arrow[from=1-5, to=3-6]
	\arrow[from=1-5, to=5-5]
	\arrow[from=2-1, to=1-2]
	\arrow[from=2-1, to=3-3, phantom, crossing over]
	\arrow[from=2-1, to=3-3]
	\arrow[from=2-1, to=6-1]
	\arrow[from=3-3, to=3-6, phantom, crossing over]
	\arrow[from=3-3, to=3-6]
	\arrow[from=3-6, to=7-6]
	\arrow[from=5-2, to=5-5]
	\arrow[from=5-5, to=7-6]
	\arrow[from=6-1, to=5-2]
	\arrow[from=6-1, to=7-3]
	\arrow[from=7-3, to=7-6]
	\arrow[from=3-3, to=7-3, phantom, crossing over]
	\arrow[from=3-3, to=7-3]
	\end{tikzcd}
        ~.
    \end{equation}
    The homotopy commutativity of the lateral faces of the prism is guaranteed by \autoref{prop:gauge-for-ham}, \autoref{rem:homotopy-comomentum-maps-are-homotopies} and \autoref{lem:gauge-transf-twisted-Courant-DGLAs}
%     which imply the homotopy commutativity of the following diagram:
%    \begin{displaymath}
%       \begin{tikzcd}
%           \mathfrak{K} \ar[d] \ar[r]
%           &[2em] 0 \ar[d]
%           \\
%           \X_{\ham,\sigma} \ar[r,"\iota_{\infty}\sigma"'] \ar[ur,Rightarrow,sloped,"\mathcal{h}"] \ar[d] &
%           (\Omega^0\to\cdots\to \Omega^{r-1}\to \d\Omega^{r-1}) \ar[d]
%           \\
%           \gSmall_{r,\sigma}^{\ge 0} \ar[d,leftrightarrow,sloped,"\sim"] \ar[r,hook] \ar[ur,Rightarrow,sloped,"\iota"]
%           & \gSmall_{r,\sigma} \ar[d,leftrightarrow,sloped,"\sim"]
%           \\
%           \gSmall_{r,\sigma'}^{\ge 0} \ar[r,hook] \ar[dr,Rightarrow,sloped,"\iota"']
%           & \gSmall_{r,\sigma'}
%           \\
%           \X_{\ham,\sigma'} \ar[u] \ar[r,"\iota_{\infty}\sigma'"] \ar[dr,Rightarrow,sloped,"\mathcal{h}'"']
%           &
%           (\Omega^0\to\cdots\to \Omega^{r-1}\to \d\Omega^{r-1})
%           \ar[u]
%           \\
%           \mathfrak{K} \ar[u] \ar[r]
%           & 0 \ar[u]
%        \end{tikzcd}
%        ~.
%    \end{displaymath}
%    The 
For what concerns the top face
\begin{equation}\label{eq:top-face}
       \begin{tikzcd}
            &[1em]
            \X_{\ham,\sigma}\ar[r] &
            \gSmall_{r+1,\sigma}^{\ge 0} \ar[dd,"\chi_{\beta}"]
            \\[-.75em]
            \mathfrak{K} \ar[ur,"\rho"] \ar[dr,"\rho'"'] &
            &
            \\[-.75em]
            &
            \X_{\ham,\sigma'} \ar[r] &
            \gSmall_{r+1,\sigma'}^{\ge 0}
        \end{tikzcd}
        ~,
    \end{equation}
  since the source $\mathfrak{K}$ is an ordinary Lie algebra and the target $\gSmall_{r+1,\sigma'}^{\ge 0}$ is a DGLA in non-negative degrees, an $L_\infty$-morphism between them is necessarily a DGLA morphism, and a homotopy between two such morphisms is necessarily the identity, so diagram \eqref{eq:top-face} has to strictly commute. This requirement corresponds to the identity
  \[
 \rho'(\xi)= \rho(\xi)+\Lie_{\rho(\xi)}\beta
  \]
for every $\xi \in \mathfrak{K}$. Thus we retrieve the hypothesis of \cite[Thm.~5.3]{Miti2024} ensuring the commutativity of diagram \eqref{eq:pentagonal-diagram}.
As a side remark, notice that also the  bottom face of \eqref{eq:prism-diagram} trivially commutes since $0$ is the initial $L_\infty$-algebra.

%=================================================#
\section{Outlook}
%=================================================#
In \cite{Zambon2012}, Zambon introduced higher Courant algebroids to generalize the notion of Dirac structures. 
All the constructions presented here are likely to have an immediate generalization to the higher Dirac structures case. 
Here, we focused on the multisymplectic case because, in our opinion, it represents an ideal balance between classical differential-geometry constructions (Poisson structures, Hamiltonian vector fields, Cartan calculus) and abstract homotopical algebra constructions. 
Developing the whole construction directly for higher Dirac structures and recovering the multisymplectic case as a special example would have possibly made the whole treatment more obscure, so we resisted the temptation of generality. 
Details on the general construction will appear elsewhere \cite{Basu2026}.
Here, we content ourselves with recalling the basic facts about higher Dirac structures to provide a short presentation of the expected results in the general case.

    \begin{definition}[Higher Dirac Structures]
        Let $M$ be a smooth manifold and $r\geq 0$ an integer.
        A \emph{higher Dirac structure} of degree $r$ is a subbundle $L\subset E^{(r)}=TM \oplus \wedge^r T^*M$ which is involutive with respect to the higher Courant bracket and maximally isotropic with respect to the pairing $\langle\cdot,\cdot\rangle_{+}$.
    \end{definition}
%Multisymplectic forms are a particular case of higher Dirac structures.
    \begin{example}[Multisymplectic forms as Dirac structures]
        Let $\sigma\in \Omega^{r+1}(M)$ be a differential form. 
        The graph of the bundle map $\sigma^\flat: TM\to \wedge^r T^*M$, given by $\sigma^\flat(x)=\iota_x\sigma$, is a subbundle of $E^{(r)}$; if $\sigma$ is closed, this graph is a higher Dirac structure.
    \end{example}
    Given an involutive and isotropic subbundle $L\subset E^{(r)}$, one can pinpoint a special class of differential forms called \emph{admissible forms}:
    \begin{equation}
        \label{eq:admissible-forms-Zambon}
            \Omega^{r-1}_{\adm}(M,L) = \lbrace \alpha \in \Omega^{r-1}(M) ~|~ \exists X \in \X(M) : X + \d\alpha \in \Gamma(L) \rbrace
            ~.
    \end{equation} 
Given $\sigma \in \Omega^{r-1}_{\adm}(M,L)$ the associated \emph{admissible  vector field} $X_\sigma$ is unique modulo $\Gamma(L \cap (TM\oplus 0))$.
One can therefore consider the following binary operation on admissible forms:
$$
    \morphism{\lbrace\blank,\blank\rbrace_{\adm}}{\Omega^{r-1}_{\adm}(M,L) \times \Omega^{r-1}_{\adm}(M,L)}{\Omega^{r-1}_{\adm}(M,L)}
{\sigma_1,\sigma_2}
{\iota_{X_{\sigma_1}} \d \sigma_2} 
$$
which is skew-symmetric and satisfies the Jacobi identity up to exact forms.
In \cite[Thm 6.7]{Zambon2012} it is shown that the bracket on admissible forms is part of an $L_\infty$-algebra $\Zambon[L]{r-1}$ associated with the Dirac structure $L$ in a way similar to the Rogers' construction.
Therefore, one could naturally ask to what extent the diagram \eqref{eq:final-diagram} can be generalized to higher Dirac structures.
In other words, the following questions arise:
\begin{itemize}
    \item[Q1)] Does the recognition principle used in \cite{Fiorenza2014} to characterize Rogers' $L_\infty$-algebra as a homotopy fiber also applies to $\Zambon[L]{r-1}$?
     \item[Q2)] Can one use a higher Dirac structure $L$ to twist a Courant higher algebroid, obtaining an $L_\infty$-algebra $\Courant[L]{r-1}$
      \item[Q3)] Has one natural distinguished nontrivial $L_\infty$-morphism $\Courant[L]{r-1}\to \Courant[L]{r}$
      \item[Q4)] Can one apply the above homotopical reasoning to produce an $L_\infty$-morphism from $\Zambon[L]{p-1}$ to the twisted higher Courant $L_\infty$-algebra $\Courant[L]{r-1}$? 
\end{itemize}
 Similar questions have been raised in \cite[\S 9]{Zambon2012}. 
 %Preliminary results on the application of the homotopical machinery to this setting are currently in progress \cite{Basu2026}.
\\
The generalization to higher Dirac structures could be relevant in several contexts.
For instance, they can be used to devise a higher analogue of the group valued moment maps appearing in the study of quasi-Hamiltonian $G$-spaces, and to read them in terms of Dirac structures, extending to the higher case the construction by  Ikeda \cite{Ikeda2025}  (see also \cite{Hirota2024,Hirota2025a,Hirota2025b}).
Another possibile development is to connect back to the graded geometry approach used by Zambon in defining the higher Courant $L_\infty$-algebras. Specifically, it is to be expected that the small DGLAs $\g_{r,\sigma}$ and $\g_{r+1}$ we used to exhibit the distinguished morphism $\Courant[\sigma]{r-1}\to \Courant{r}$ are purely auxiliary and that the whole construction can be expressed in terms of the graded geometry of shifted cotangent bundles. Moreover, the one-to-one correspondence between (twisted) higher Dirac structures in  $E^{(r)}$ and Lagrangian $Q$-submanifolds of $T^*[r+1]T[1]M$  \cite[Thm. 4.8]{Cueca2021} suggests this geometric description could have a natural interpretation in terms of higher Dirac structures.
A more speculative outlook concerns possible applications to Poisson sigma models (see, e.g., \cite{Ikeda2025}).

%#########################################################################%
% APPENDICES
%#########################################################################%
\appendix
%#########################################################################%

%=================================================#
\section{A primer on \texorpdfstring{$L_\infty$}{L-infinity}-algebras}
\label{appendix:linf}
%=================================================#
%
In this appendix, we briefly recall the basic notions of $L_\infty$-structures (algebras, morphisms, and homotopies) and recall the notion of homotopy fibers used throughout the paper. Our notations are borrowed from \cite{Fiorenza2009,Miti2021} and references therein. 
%This section contains nothing substantially new, except for the form in which the material is presented; we 
We include a few proofs for the sake of self-consistency. We address the reader to \cite{Manetti2022} for a comprehensive treatment of the subject. See also \cite{Kraft2024} for a recent review on the topic.
%
%-.-.-.-.-.-.-.-.-.-.-.-.-.-.-.-.-.-.-.-.-.-.-.-.-.-.-.-.-.-.-+
\subsection{\texorpdfstring{$L_\infty$}{L-infinity}-structures as Maurer-Cartan elements}
%-.-.-.-.-.-.-.-.-.-.-.-.-.-.-.-.-.-.-.-.-.-.-.-.-.-.-.-.-.-.-+
Let $V, W$ be two graded vector spaces.
We denote the vector space of homogeneous, graded, skew-symmetric, multilinear maps from $V$ to $W$ with arity $a \geq 1$ and degree $d \in \mathbb{Z}$ by the symbol
\begin{equation}\label{eq:biggraded-Hom-arity-degree}
    \underline{\Hom}^{a,d}\big(V, W\big) := \Hom_{\gVect}\big(\wedge^a V, W[d]\big)~.
\end{equation}
The direct sum of these vector spaces forms a bigraded vector space that we will denote by $\underline{\Hom}^{\bullet,\bullet}\big(V, W\big)$ or simply by $\underline{\Hom}\big(V, W\big) $. 
On such space one can define the so-called \emph{Nijenhuis-Richardson product}~\cite{Nijenhuis1967} as follows:
\begin{equation}\label{eq:RN-product}
    \morphism{\ca}{\underline{\Hom}^{a_1,d_1}\big(V, W\big) \otimes \underline{\Hom}^{a_2,d_2}\big(V,V\big)}{\underline{\Hom}^{a_1 + a_2 - 1, d_1 + d_2}(V, W)}
    {(\mu,\nu)}{(-1)^{(a_1-1)d_2}\,\mu \circ \big(\nu \otimes \mathbb{1}_{a_1-1}\big) \circ P_{a_2,a_1-1}}
    ~,
\end{equation}
    where $\mathbb{1}_{i}$ denotes the identity endomorphism on $V^{\otimes i}$, $P_{j,k}$ is the signed unshuffle\footnote{
        The group of $(j,k)$-unshuffles $\mathrm{Sh}(k_1,\dots,k_\ell)$ is the subgroup of permutations of $j+k$ elements that preserve the order of the first $j$ and the last $k$ elements, i.e. $\sigma(i)<\sigma(i+1),~ \forall i\neq j$.
    } 
    permutation operator
    \begin{equation}\label{eq:unshuffle-permutator}
        P_{j,k}(x_1 \wedge \cdots \wedge x_{j+k}) = \sum_{\sigma \in \ush{j,k}} (-1)^\sigma \varepsilon(\sigma,x)~ x_{\sigma(1)} \wedge \cdots \wedge x_{\sigma(j+k)}~,\qquad \forall x_i \in \g~,
    \end{equation}
    $(-1)^\sigma$ is the signature of the permutation $\sigma$
    and $\varepsilon(\sigma,x)$ is the Koszul sign associated to the permutation $\sigma$ acting on the graded elements $x_1, \dots, x_{j+k} \in \g$.
    \begin{remark}[Sign prefactors]
        The appearance of the sign prefactor $(-1)^{(a_1-1)d_2}$ in \eqref{eq:RN-product} is due to the fact that we are considering graded skew-symmetric multilinear maps.
        In fact, the original definition of Nijenhuis-Richardson product in \cite{Nijenhuis1967} is given for graded symmetric multilinear maps and does not include such a sign and the appeareance of that sign is a manifestation of the \emph{decalage isomorphism} (See \cite[\S 10.6]{Manetti2022} or \cite[\S 1]{Miti2021}).
    \end{remark}
    \begin{remark}\label{rmk:associator-Nijenhuis-Richardson}
        It is important to notice that the Nijenhuis-Richardson product $\ca$ is in general not associative.
        More explicitly, given any three multilinear operators $\mu_a \in \underline{\Hom}^{a,\deg(\mu_a)}(V,W)$, $\mu_b \in \underline{\Hom}^{b,\deg(\mu_b)}(V,V)$, $\mu_c \in \underline{\Hom}^{c,\deg(\mu_c)}(V,V)$ the corresponding associators
        $$
            \alpha(\ca; \mu_a,\mu_b,\mu_c)  =(\mu_a \ca \mu_b)\ca \mu_c - \mu_a \ca (\mu_b \ca \mu_c)
        ~,
        $$
        results
		\begin{align*}
			\alpha(\ca; \mu_a,\mu_b,\mu_c) 
			=&~
			(-1)^{\mathcal{s}}
			\mu_a \circ \left((\mu_b\otimes\mu_c \circ P_{b,c})\otimes \Unit_{a-2}\right) \circ P_{b+c,a-2}		
			~.
		\end{align*}
		where the latter sign prefactor is given by:
		\begin{displaymath}
			{\mathcal{s}}={{\deg(\mu_c)(b+a)} +{\deg(\mu_b)(a-1)} +b(c+1)}
			~.
		\end{displaymath}
        In particular $\alpha(\ca; \mu_a,\mu_b,\mu_c)=0$ whenever $a=1$ and when $a=2$ it gives
        \begin{displaymath}
            \alpha(\ca; \mu_2,\mu_b,\mu_c) = (-1)^{\deg(\mu_b) + b\,(\deg(\mu_c)+ c+1)} \mu_2 \circ \left((\mu_b\otimes\mu_c \circ P_{b,c})\right)~.
        \end{displaymath}
        Moreover, the associator enjoys the following graded symmetry on the two right-most entries:
        \begin{displaymath}
            \alpha(\ca; \mu_a,\mu_b,\mu_c) = (-1)^{(\deg(\mu_b)+b-1)(\deg(\mu_c) + c -1 )} \alpha(\ca; \mu_a,\mu_c,\mu_b)~.
        \end{displaymath}
        We refer to \cite[Appendix B]{Miti2021} for further details.
    \end{remark}    
    Let us now focus on the case $W = V$. Although the Nijenhuis-Richardson product is not associative, it can be shown to be pre-Lie.
    \begin{definition}[Nijenhuis-Richardson algebra]\label{def:NR-algebra}
        The \emph{Nijenhuis-Richardson algebra} over $V$ is the Lie algebra $M^{\textrm{skew}}(V)$ given by the graded vector space of multibrackets
        \begin{displaymath}
            M^{\textrm{skew}}(V)^k := \bigoplus_{a+d-1 =k} \underline{\Hom}^{a,d}(V, V)~,
        \end{displaymath}
        together with the Lie bracket $\lbrack\cdot,\cdot\rbrack_\ca$ induced by the Nijenhuis-Richardson (pre-Lie) product $\ca$ defined in \eqref{eq:RN-product}.
    \end{definition}
    \begin{definition}[\texorpdfstring{$L_\infty$}{L-infinity}-algebra structure {\cite{Lada1993,Lada1995}}]
        An \emph{\(L_\infty\)-algebra structure} on the graded vector space \(L\) is given by a collection of homogeneous graded skew-symmetric multilinear maps
        \begin{displaymath}
            \mu_k: \wedge^k L \to L[2-k]~,
        \end{displaymath}
        for \(k \geq 1\), such that $\mu=\sum_{k \geq 1} \mu_k \in M^{\textrm{skew}}(L)$ is a Maurer-Cartan element in the  the Nijenhuis-Richardson algebra.
        A pair $(L, \mu)$ is called an \emph{\(L_\infty\)-algebra}.
    \end{definition}
    \begin{example}
        A DGLA is an $L_\infty$-algebra with $\mu_k=0$ for any $k\geq 3$.
        In this case, the Maurer-Cartan equation reads
        \begin{align*}
               & \mu_1 \ca \mu_1= 0 &\qquad& {\text{in arity $1$,}}
                \\
                &\mu_1 \ca \mu_2 = - \mu_2\ca \mu_1 &\qquad& {\text{in arity $2$,}}
                \\
                &\mu_2\ca\mu_2 = 0 &\qquad& {\text{in arity $3$.}}
        \end{align*}
    \end{example}
\begin{remark}\label{rmk:coalgebraic-approach-Linfty-algebras}
    A more conceptual characterization of $L_\infty$-algebras is naturally expressed in the language of differential graded cocommutative coalgebras.  
    Given a graded vector space $L$, one can consider the reduced cofree cocommutative coalgebra $C(L[1]) = S^{\ge1}(L[1])$ cogenerated by the suspension $L[1]$ and endowed with its canonical coproduct. 
    An $L_\infty$-structure on $L$, specified by a family of graded-symmetric multilinear maps $\mu_k:S^k(L[1])\to L[1]$, can be equivalently encoded by a single degree-$1$ coderivation 
    $$Q_\mu : C(L[1]) \longrightarrow C(L[1])~,$$
    obtained as the unique lift of the shifted sum $\mu = \mu_1 + \mu_2 + \cdots$.  
    The higher Jacobi identities are then compactly expressed by the condition $Q_\mu^2=0$, so that an $L_\infty$-algebra $(L,\mu)$ corresponds precisely to a differential graded cocommutative coalgebra $(C(L[1]),Q_\mu)$.  
    In this setting, it is particularly easy to make sense of morphisms between $L_\infty$-algebras as coalgebra maps interwining the differential (see \autoref{rmk:coalgebraic-approach-Linfty-morphisms}) and thus completely define the category $\Linfty$.
\end{remark}

%-.-.-.-.-.-.-.-.-.-.-.-.-.-.-.-.-.-.-.-.-.-.-.-.-.-.-.-.-.-.-+
\subsection{\texorpdfstring{$L_\infty$}{L-infinity}-morphisms between DGLAs as Maurer-Cartan elements}
%-.-.-.-.-.-.-.-.-.-.-.-.-.-.-.-.-.-.-.-.-.-.-.-.-.-.-.-.-.-.-+
    A differential graded Lie algebra (DGLA) is a particular case of an $L_\infty$-algebra where only the unary and binary brackets are non-trivial.
    As such, DGLAs form a subcategory of the category of $L_\infty$-algebras.
    However, this subcategory is not full, as there exist $L_\infty$-morphisms between DGLAs that are not strict morphisms of DGLAs.
    %\antonio{inserire cappello in cui si motiva perché si può parlare di morfismi L-infinity tra DGLAs}
%
    \begin{definition}[\texorpdfstring{$L_\infty$}{L-infinity}-morphisms between DGLAs {\cite[\S 4]{Kontsevich2003}}]
    \label{def:Linfty-morphism-DGLAs}
    Let $\g$ and $\h$ be two differential graded Lie algebras.
    An \emph{\(L_\infty\)-morphism} from \(\g\) to \(\h\) consists of a collection of skew-symmetric multilinear maps
    \begin{displaymath}
        \Phi_k : \g^{\otimes k} \to \h[1-k]~,
    \end{displaymath}
    for \(k \geq 1\), such that the following compatibility conditions holds
    \begin{equation}\label{eq:Linfty-morphism-DGLAs}
        \Phi_{k} \ca \d_\g + \Phi_{k-1} \ca [\cdot,\cdot]_\g =   \d_\h \circ \Phi_{k} + [\cdot,\cdot]_\h\circ {\Sop_{2,k}(\Phi)} ~,
    \end{equation}
    where \(\ca\) denotes the Nijenhuis-Richardson composition of multilinear maps as per \eqref{eq:RN-product}, 
    the operator $\Sop_{2,k}(\Phi)$ is given by all the possible ways to tensor multiply two components of $\Phi$ to give an operator of arity $k$, i.e., 
    \begin{displaymath}
        \Sop_{2,k}(\Phi) = 
        \sum_{1\leq i \leq \lfloor \frac{k}{2}\rfloor}  
        { (-1)^{\deg(\Phi_{i})}} \left(\Phi_{i} \otimes \Phi_{k-i}\right)\circ P^{<}_{i,k-i} ~,
    \end{displaymath}
    %
    %\antonio{Ho confrontato in varie fonti. Il segno ${ \pm}$ è proprio $(-1)^{|\Phi_{i}|}=(-1)^{1-i}$, e $\d_\h$ appare a destra con il più!}
    %
    and $P^{<}_{i,k-i}$ is a modified version of \eqref{eq:unshuffle-permutator} that runs only over all ordered $(i,k-i)$-unshuffles\footnote{The set of ordered unshuffles $\ush{(i,j)}^<$ is a subgroup of $\ush{(i,j)}$ given by those unshuffles $\sigma$ such that $\sigma(1)<\sigma(i+1)$ whenever $i=j$.}, i.e.
    \begin{displaymath}
        P^{<}_{i,j}= \begin{cases}
            \mathbb{1}_{i+j} & \text{if } i=j~, \\
            P_{i,j} & \text{if } i\neq j~.
        \end{cases}
    \end{displaymath}
    We denote with $\Hom_\infty(\g,\h)$ the set of all \(L_\infty\)-morphisms between the two DGLAs \(\g\) and \(\h\).
\end{definition}

\begin{remark}\label{rmk:coalgebraic-approach-Linfty-morphisms}
    Within the coalgebraic viewpoint of \autoref{rmk:coalgebraic-approach-Linfty-algebras}, given two $L_\infty$-algebras $(L,\mu)$ and $(L',\mu')$, an $L_\infty$-morphism $\Phi:L\to L'$ is simply a coalgebra morphism 
    $$
        F:C(L[1]) \longrightarrow C(L'[1])
    $$
    compatible with the codifferentials, that is, $F\circ Q_\mu = Q_{\mu'}\circ F$.  
    Since $C(L[1])$ is cofreely cogenerated by $L[1]$, such a morphism is entirely determined by its Taylor coefficients $F_n^1:S^n(L[1])\to L'[1]$, which satisfy a hierarchy of quadratic relations encoding the homotopy nature of the $L_\infty$-structure.  
    Ordinary DGLA morphisms appear as the \emph{strict} case, in which all higher Taylor components vanish.
\end{remark}

\begin{remark}\label{rmk:Linfty-morphism-DGLAs}
    The first instances of condition \eqref{eq:Linfty-morphism-DGLAs}, involving only components up to arity three, read as follows:
    \begin{align}
        \Phi_1 \circ \d_\g (v) &=~ \d_\h \circ \Phi_1 (v)  ~,
        \label{eq:Linfty-morphism-DGLAs-arity-1}
        \\[.5em]
                \Phi_1([v_1,v_2]_\g) - \Phi_2\left(\d_\g(v_1),v_2\right) -(-1)^{|v_1|} \Phi_2(v_1, \d_\g v_2) &=~  [\Phi_1(v_1),\Phi_2(v_2)]_\h + \d_\h \left( \Phi_2(v_1,v_2)\right) ~,
        \label{eq:Linfty-morphism-DGLAs-arity-2}
    \end{align}
    \begin{align}
        &\phantom{=~}\Phi_3\left( \d_\g(v_1),v_2,v_3\right) -(-1)^{|v_1||v_2|} \Phi_3\left(\d_\g(v_2),v_1, v_3\right) +(-1)^{|v_3|(|v_1|+|v_2|)} \Phi_3\left(\d_\g(v_3),v_1, v_2 \right) 
        \label{eq:Linfty-morphism-DGLAs-arity-3}
         \\
        &\phantom{=~}+ \Phi_2([v_1,v_2]_\g,v_3) -(-1)^{|v_2||v_3|} \Phi_2([v_1,v_3]_\g,v_2) +(-1)^{|v_1|(|v_2|+|v_3|)} \Phi_2( [v_2,v_3]_\g,v_1) =
        \notag \\
        &=~ 
        \d_\h\, \Phi_3(v_1,v_2,v_3) +
        \notag \\
        &\phantom{=~}+
        (-1)^{|v_1|} [\Phi_1(v_1),\Phi_2(v_2,v_3)]_\h -
        (-1)^{(|v_1+1)||v_2|} [\Phi_1(v_2),\Phi_2(v_1,v_3)]_\h +
        (-1)^{(|v_1|+|v_2|+1)|v_3|} [\Phi_1(v_3),\Phi_2(v_1,v_2)]_\h
        \notag 
    \end{align}
    for any homogeneous elements $v, v_1, v_2, v_3 \in \g$.
\end{remark}

    Also, $L_\infty$-morphisms between DGLAs can be characterized as Maurer-Cartan elements of a suitable DGLA (see, for instance, \cite{Fiorenza2007}), as we now recall.
    Let $\g$ and $\h$ be two DGLAs; one can naturally endow the bigraded vector space of multilinear maps $\underline{\Hom}(\g,\h)$ with a binary bracket
    \begin{equation}\label{eq:bracket-Hom-DGLAs}
        \morphism{[\blank,\blank]_{\underline{\Hom}}}{\underline{\Hom}^{a_1,d_1}({\g},{\h})\otimes \underline{\Hom}^{a_2,d_2}({\g},{\h})}{\underline{\Hom}^{a_1+a_2,d_1+d_2}({\g},{\h})}
        {f,g}
        {-(-1)^{d_1 (a_2 + d_2)}~\lbrack\cdot,\cdot\rbrack_\h \circ (f \otimes g) \circ P_{a_1,a_2}}~,
    \end{equation}
    % {{-}(-1)^{d_1 (a_2 + d_2)}}
%    \antonio{
%        Cambio il segno da
%        $$ [f,g]_{\underline{\Hom}} = {\color{purple}(-1)^{a_1(a_2+ d_2)}} [\cdot,\cdot]_\h \circ (f \otimes g) \circ P_{a_1,a_2}$$
%        a
%        \begin{align*}
%            [f,g]_{\underline{\Hom}} 
%            =&~ {(-1)^{d_1( a_2 + d_2)}} [\cdot,\cdot]_\h \circ ( g \otimes f) \circ P_{a_2,a_1}
%        \end{align*}
%        usando il fatto che
%        \begin{displaymath}
%            f \otimes g = -(-1)^{a_1 a_2}(-1)^{d_1 d_2} \mathrm{C}_{(2)} \circ g \otimes f \circ \mathrm{C}_{(a_1+a_2)}^{a_1}~,
%        \end{displaymath}
%        Dove $\mathrm{C}_{(n)}$ è l'operatore che performa la permutazione ciclica di $n$ elementi introdotto nell'appendice di \cite{Miti2021}, si verifica che è ancora graded skew.
%
%        Aggiungo l'ulteriore segno ${-}$ affinché, in \autoref{lem:Linfty-morphisms-as-MC-elements}, compaia l'equazione giusta.
%
%    }
    a vertical differential (raising the degree of the given multilinear map $f$)
    \begin{equation}
        \label{eq:vertical-differential-Hom-DGLAs}
        \morphism{\d_{(0,1)}}{\underline{\Hom}^{a,d}({\g},{\h})}{\underline{\Hom}^{a,d+1}({\g},{\h})}
        {f}
         { f \ca \d_\g -{ (-1)^{a+d-1}}\d_\h \ca f}~,
  \end{equation}
%    \antonio{Il segno è scelto opportunamente. Quando $f$ è una componente di un morfismo $L_\infty$, allora $a+d-1=0$.}
%    \antonio{Cambio di notazione rispetto a \cite{Fiorenza2009} dove si definiva:
 %       \[
%        d_{1,0}^{}\colon
%        \underline{\Hom}^{p,q}({\mathfrak g},{\mathfrak h})\to \underline{\Hom}^{p+1,q}({\mathfrak g},{\mathfrak h}
%        )\]
%        given by
%        \[
%            (d_{1,0}^{}{f})(\gamma_1^{}\wedge\cdots\wedge\gamma_q^{})=d_{\mathfrak h}(f(\gamma_1^{}\wedge\cdots\wedge\gamma_q^{}))+\sum_i
%            \pm f(\gamma_1\wedge\cdots
%            \wedge d_{\mathfrak g}\gamma_i\wedge
%            \cdots\wedge\gamma_{q+1}^{})
%        \]
%    }    
    and a horizontal differential (raising the arity of the given multilinear map $f$)
    \begin{equation}
        \label{eq:horizontal-differential-Hom-DGLAs}
        \morphism{\d_{(1,0)}}{\underline{\Hom}^{a,d}({\g},{\h})}{\underline{\Hom}^{a+1,d}({\g},{\h})}
        {f}
        {f \ca [\cdot,\cdot]_\g}~.
    \end{equation}
    \begin{lemma}
        The bigraded vector space $\underline{\Hom}(\g,\h)$ endowed with the bracket $[\cdot,\cdot]_{\underline{\Hom}}$ in \eqref{eq:bracket-Hom-DGLAs} and the differentials $\d_{(0,1)}$ and $\d_{(1,0)}$ in \eqref{eq:vertical-differential-Hom-DGLAs} and \eqref{eq:horizontal-differential-Hom-DGLAs} is a bidifferential graded Lie algebra.
        In particular, the total differential $\d_{\textrm{tot}} = \d_{(1,0)} + \d_{(0,1)}$ squares to zero and is a derivation of the bracket $[\cdot,\cdot]_{\underline{\Hom}}$.
        \label{lem:Hom-DGLAs-is-bidgla}
    \end{lemma}
    \begin{proof}
        We need to show that $\d_{(1,0)}$ and $\d_{(0,1)}$ are differentials that anticommute with each other and that they are derivations of the bracket $[\cdot,\cdot]_{\underline{\Hom}}$.
        Recall that $\d_\h \ca f = \d_\h \circ f$. 
        For the vertical differential, we have
        \begin{align*}
            \d_{(0,1)} \d_{(0,1)} (f) 
            &=~
            (\d_{(0,1)} (f)) \ca \d_\g -{(-1)^{a+d}} \d_\h \ca (\d_{(0,1)} (f))
            \\
            &=~
            (f \ca \d_\g - {(-1)^{a+d-1}}\d_\h \ca f) \ca \d_\g -{(-1)^{a+d}} \d_\h \ca (f \ca \d_\g -{(-1)^{a+d-1}} \d_\h \ca f)
            \\
            &=~
            (f \ca \d_\g) \ca \d_\g +{(-1)^{a+d}}\left( (\d_\h \ca f) \ca \d_\g - \d_\h \ca (f \ca \d_\g)\right) - \d_\h \ca (\d_\h \ca f)
            \\
            &=~
            f \ca \cancel{(\d_g \ca \d_g)} + \alpha(\ca; f, \d_\g, \d_\g) +{(-1)^{a+d}}\cancel{\alpha(\ca; \d_\h, f, \d_\g)} +  \cancel{(\d_\h \ca \d_\h)} \ca f - \cancel{\alpha(\ca; \d_\h, \d_\h, f)}
            \\
            &=~
            \cancel{ \alpha(\ca; f, \d_\g, \d_\g)} =0 ,
        \end{align*}
        where the last cancellation follows from the graded symmetry of the associator in \autoref{rmk:associator-Nijenhuis-Richardson} which implies $\alpha(\ca; f, \d_\g, \d_\g) =-\alpha(\ca; f, \d_\g, \d_\g)$
        \\
        For the horizontal differential, we have
        \begin{align*}
            \d_{(1,0)} \d_{(1,0)} (f) 
            &=~
            \d_{(1,0)} (f \ca [\cdot,\cdot]_\g)
            \\
            &=~
            (f \ca [\cdot,\cdot]_\g) \ca [\cdot,\cdot]_\g
            \\
            &=~
            f \ca \cancel{([\cdot,\cdot]_\g \ca [   \cdot,\cdot]_\g)} + \cancel{\alpha(\ca; f, [\cdot,\cdot]_\g, [\cdot,\cdot]_\g)}
        \end{align*}
        where the first cancellation follows from the Jacobi identity in $\g$ and the second from the symmetry of the associator in \autoref{rmk:associator-Nijenhuis-Richardson}.
        \\
        By direct inspection, one can check that vertical and horizontal differentials anticommute:
        \begin{align*}
            \d_{(1,0)} \d_{(0,1)} (f) 
            &=~
            \d_{(1,0)} (f \ca \d_\g - {(-1)^{a+d-1}} \d_\h \ca f)
            \\
            &=~
            \left(f \ca \d_\g  - {(-1)^{a+d-1}} (\d_\h \ca f)\right) \ca [\cdot,\cdot]_\g
            \\
            \\
            \d_{(0,1)} \d_{(1,0)} (f) 
            &=~
            \d_{(0,1)} (f \ca [\cdot,\cdot]_\g)
            \\
            &=~
            (f \ca [\cdot,\cdot]_\g) \ca \d_\g - {(-1)^{a+d}} \d_\h \ca (f \ca [\cdot,\cdot]_\g)
            \\
            &=~
            f \ca ([\cdot,\cdot]_\g \ca \d_\g) + \alpha(\ca; f, [\cdot,\cdot]_\g, \d_\g) - {(-1)^{a+d}}  (\d_\h \ca f) \ca [\cdot,\cdot]_\g \pm \cancel{\alpha(\ca; \d_\h, f, [\cdot,\cdot]_\g)}
            \\
            &=~
            -f \ca (\d_\g\ca[\cdot,\cdot]_\g) - \alpha(\ca; f, \d_\g, [\cdot,\cdot]_\g) - {(-1)^{a+d}}  (\d_\h \ca f) \ca [\cdot,\cdot]_\g 
            \\
            &=~
            -(f\ca \d_\g) \ca [\cdot,\cdot]_\g 
            \cancel{+\alpha(\ca; f, \d_\g, [\cdot,\cdot]_\g)
            -\alpha(\ca; f, \d_\g, [\cdot,\cdot]_\g)} 
            - {(-1)^{a+d}}  (\d_\h \ca f) \ca [\cdot,\cdot]_\g 
            \\
            &=~
            -\d_{(1,0)} \d_{(0,1)} (f)
        \end{align*}
        therefore $\d_{\textrm{tot}}=\d_{(1,0)} + \d_{(0,1)}$ is a honest differential.
        \\
        Finally, we need to show that both differentials are derivations of the bracket $[\blank,\blank]_{\underline{\Hom}}$ i.e. that for any $f \in \underline{\Hom}^{a_1,d_1}(\g,\h)$ and $g \in \underline{\Hom}^{a_2,d_2}(\g,\h)$ it holds
        \begin{align*}
            \d_{\textrm{tot}} \circ [\blank,\blank]_{\underline{\Hom}} (f,g) 
            &=~
            %\pm 
            [\blank,\blank]_{\underline{\Hom}}(\d_{\textrm{tot}} f, g) + 
            (-1)^{a_1+d_1} [\blank,\blank]_{\underline{\Hom}}(f, \d_{\textrm{tot}} g)~.
        \end{align*}
%        where $|f|=a_1 + d_1$ and $|g|=a_2 + d_2$ are the total degrees of $f$ and $g$ respectively.        \\
            For the horizontal differential, we have
            \begin{align*}
             \d_{(1,0)} \circ [\blank,\blank]_{\underline{\Hom}} (f,g)
            &=~
            \left( [f,g]_{\underline{\Hom}}\right) \ca [\cdot,\cdot]_\g
            \\
            &=~
            \left({{-}(-1)^{d_1 (a_2 + d_2)}}~
            \lbrack\cdot,\cdot\rbrack_\h \circ (f \otimes g) \circ P_{a_1,a_2} \right)\ca [\cdot,\cdot]_\g
            \\
            &=~
            {{-}(-1)^{d_1 (a_2 + d_2)}}{(-1)^{(a_1+a_2-1)0}}~\left(\lbrack\cdot,\cdot\rbrack_\h \circ (f \otimes g) \circ P_{a_1,a_2} \right) \circ \left([\cdot,\cdot]_\g \otimes \mathbb{1}_{a_1 + a_2 -1}\right) \circ P_{2,a_1 + a_2 -1}
            \end{align*}
            Noticing that for any graded skew-symmetric $n$-ary map $\mu_n$ one has (cf. \cite[Lem. C.3.3]{Miti2021})
        \begin{displaymath}
            (P_{a_1,a_2}) \circ (\mu_n \otimes \mathbb{1}_{a_1+a_2-1}) \circ (P_{n,a_1+a_2-1}) 
	        =
        	(\mu_n \otimes \mathbb{1}_{a_1+a_2-1}) \circ
			(P_{n,a_1-1,a_2}) +
	    	(-1)^{a_1(n+1)}~
		    \left(\mathbb{1}_{a_1}\otimes
		    \mu_n\otimes \mathbb{1}_{a_2-1}
		    \right) \circ
		    (P_{a_1,n,a_2-1})	
		    ~,
        \end{displaymath}
        we can rewrite the last term as the sum of two terms.
        The first one reads
        \begin{align*}
            &{{-}(-1)^{d_1 (a_2 + d_2)}}~
            \lbrack\cdot,\cdot\rbrack_\h \circ (f \otimes g) \circ 
            \left(
                \lbrack\cdot,\cdot\rbrack_\g \otimes \mathbb{1}_{a_1+a_2-1}
            \right)
            \circ P_{2,a_1-1, a_2}
            \\
            =&~
            {{-}(-1)^{d_1 (a_2 + d_2)}}~
            \lbrack\cdot,\cdot\rbrack_\h \circ \left( (f \ca [\cdot,\cdot]_\g) \otimes g\right) \circ P_{a_1+1,a_2}
            \\
            =&~
            {{-}(-1)^{d_1 (a_2 + d_2)}}~
            \lbrack\cdot,\cdot\rbrack_\h \circ \left( (\d_{(1,0)} f) \otimes g\right) \circ P_{a_1+1,a_2} 
            \\
            =&~
            %{(-1)^{a_2+d_2}}
            [\blank,\blank]_{\underline{\Hom}}(\d_{(1,0)} f, g)
        \end{align*}
        The second one reads
        \begin{align*}
            &{{-}(-1)^{d_1 (a_2 + d_2)}}{(-1)^{a_1(2 + 1)}}~
            \left(\lbrack\cdot,\cdot\rbrack_\h \circ (f \otimes g) \circ 
            \left(\mathbb{1}_{a_1}\otimes
        \lbrack\cdot,\cdot\rbrack_\g\otimes \mathbb{1}_{a_2-1}
            \right) \circ P_{a_1, 2, a_2 -1}\right)
            \\
            =&~
            {{-}(-1)^{d_1 (a_2 + d_2)}}{(-1)^{a_1}}~
            \lbrack\cdot,\cdot\rbrack_\h \circ \left( f \otimes (g \ca [\cdot,\cdot]_\g)\right) \circ P_{a_1, a_2+1}
            \\
            =&~
            {{-}(-1)^{d_1 (a_2 + d_2)}}{(-1)^{a_1}}~
            \lbrack\cdot,\cdot\rbrack_\h \circ \left( f \otimes (\d_{(1,0)} g)\right) \circ P_{a_1, a_2+1}
            \\
            =&~
            {(-1)^{a_1+d_1}}
            [\blank,\blank]_{\underline{\Hom}}(f, \d_{(1,0)} g)~.
        \end{align*}

        For the vertical differential, we have
        \begin{align*}
            \d_{(0,1)} \circ [\blank,\blank]_{\underline{\Hom}} (f,g) 
            &=~
            ( [f,g]_{\underline{\Hom}}) \ca \d_\g - {(-1)^{a_1+a_2+d_1 + d_2 -1}} \d_\h \ca ( [f,g]_{\underline{\Hom}}  )
            \\
            &=~
            {{-}(-1)^{d_1 (a_2 + d_2)}}~\lbrack\cdot,\cdot\rbrack_\h \circ (f \otimes g) \circ P_{a_1,a_2} \ca \d_\g
            \\
            &-
            {(-1)^{a_1+a_2+d_1 + d_2 -1}} \d_\h \ca \left( {{-}(-1)^{d_1 (a_2 + d_2)}}~\lbrack\cdot,\cdot\rbrack_\h \circ (f \otimes g) \circ P_{a_1,a_2} \right)
            \\
            &=~
                        {{-}(-1)^{d_1 (a_2 + d_2)}}{(-1)^{(a_1+a_2-1)1}}~\left(\lbrack\cdot,\cdot\rbrack_\h \circ (f \otimes g) \circ P_{a_1,a_2} \right) \circ \left(\d_\g \otimes \mathbb{1}_{a_1 + a_2-1}\right) \circ P_{1,a_1 + a_2 -1}
            \\
            &-
            {(-1)^{a_1+a_2+d_1 + d_2 -1}} 
            {{-}(-1)^{d_1 (a_2 + d_2)}}~
            \d_\h \circ \lbrack\cdot,\cdot\rbrack_\h \circ (f \otimes g) \circ P_{a_1,a_2}
        \end{align*}
        The first summand can be rewritten as the sum of two terms using the same trick as before 
        \begin{align*}
            &~
            {{-}(-1)^{d_1 (a_2 + d_2)}}{(-1)^{(a_1+a_2-1)1}}~
            \lbrack\cdot,\cdot\rbrack_\h \circ (f \otimes g)
            \left(
                \d_\g \otimes \mathbb{1}_{a_1 + a_2-1}
            \right) 
            \circ P_{1,a_1 -1, a_2}
            \\
            &+
            {{-}(-1)^{d_1 (a_2 + d_2)}}{(-1)^{(a_1+a_2-1)1}}~
            \lbrack\cdot,\cdot\rbrack_\h \circ (f \otimes g)
            \left(
                \mathbb{1}_{a_1}\otimes \d_\g \otimes \mathbb{1}_{a_2-1}
            \right) 
            \circ P_{a_1, 1, a_2-1}
            \\
            =&~
            {{-}(-1)^{d_1 (a_2 + d_2)}}{(-1)^{(a_1+a_2-1)1}}~
            {(-1)^{a_1-1}}~
            {(-1)^{d_2}}~
            \lbrack\cdot,\cdot\rbrack_\h \circ \left( (f \ca \d_\g) \otimes g\right) \circ P_{a_1+1,a_2} 
            \\
            &+
            {{-}(-1)^{d_1 (a_2 + d_2)}}{(-1)^{(a_1+a_2-1)1}}
            {(-1)^{a_2-1}}~
            \lbrack\cdot,\cdot\rbrack_\h \circ \left( f \otimes (g \ca \d_\g)\right) \circ P_{a_1, a_2+1} 
            \\
            =&~
            {(-1)^{(d_1+1) (a_2 + d_2)}}
            \lbrack\cdot,\cdot\rbrack_\h \circ \left( (f \ca \d_\g) \otimes g\right) \circ P_{a_1+1,a_2} 
            \\
            &+
            {(-1)^{d_1 (a_2 + d_2+1) +a_1+d_1}}~
            \lbrack\cdot,\cdot\rbrack_\h \circ \left( f \otimes (g \ca \d_\g)\right) \circ P_{a_1, a_2+1} 
            \\            
            =&~
            {}~
            [\blank,\blank]_{\underline{\Hom}}(f\ca \d_\g, g) + 
            {(-1)^{a_1+d_1}}~
            [\blank,\blank]_{\underline{\Hom}}(f, g \ca \d_\g)~.
        \end{align*}
        Similarly, the second summand can be rewritten, recognizing that 
        $$
            \d_\h \circ [\cdot,\cdot]_\h = 
            \d_\h \ca [\cdot,\cdot]_\h = 
            - [\cdot,\cdot]_\h \ca \d_\h = 
            [\cdot,\cdot]_\h \circ (\d_\h \otimes \mathbb{1} ) \circ P_{1,1} 
            = [\cdot,\cdot]_\h \circ (\d_\h \otimes \mathbb{1} + \mathbb{1} \otimes \d_\h)
        $$
        hence
        \begin{align*}
          -&~
            {(-1)^{a_1+a_2+d_1 + d_2 -1}} 
            {{-}(-1)^{d_1 (a_2 + d_2)}}~
            \d_\h \circ \lbrack\cdot,\cdot\rbrack_\h \circ (f \otimes g) \circ P_{a_1,a_2} 
            \\
            =&~ 
                      -
            {(-1)^{a_1+a_2+d_1 + d_2 -1}} 
            {{-}(-1)^{d_1 (a_2 + d_2)}}~
            [\cdot,\cdot]_\h \circ (\d_\h \otimes \mathbb{1} + \mathbb{1} \otimes \d_\h) \circ (f \otimes g) \circ P_{a_1,a_2}
            \\
            =&~ 
                      -
            {(-1)^{a_1+a_2+d_1 + d_2 -1}} 
            {{-}(-1)^{d_1 (a_2 + d_2)}}~
            \left\lbrace
            [\d_\h f, g]_\h + (-1)^{d_1} [f, \d_\h g]  \right\rbrace \circ P_{a_1,a_2}
            \\
            =&~ 
            { (-1)^{(d_1+1) (a_2 + d_2)} } 
            [{\color{black}(-1)^{a_1+d_1}}~ \d_h \ca f, g]_\h \circ P_{a_1,a_2}
            +
            { (-1)^{d_1 (a_2 + d_2 +1) + a_1 + d_1}} 
            [f, {\color{black}(-1)^{a_2 + d_2}}~ \d_h \ca g]_\h \circ P_{a_1,a_2}
            \\
            =&~
            {}~
            [\blank, \blank]_{\underline{\Hom}}({\color{black}(-1)^{a_1 + d_1}}~ \d_\h \ca f, g) +
            { (-1)^{a_1 + d_1}}~
            [\blank, \blank]_{\underline{\Hom}}(f, {\color{black}(-1)^{a_2 + d_2}}~ \d_\h \ca g)~.
        \end{align*}
        Putting all together, we have that the vertical differential is also a derivation of the bracket $[\cdot,\cdot]_{\underline{\Hom}}$.
    \end{proof}
 
%    \antonio{Cambio di notazione rispetto a \cite{Fiorenza2009} dove si definiva:
%    \[
%        d_{0,1}^{}\colon
%        \underline{\Hom}^{p,q}({\mathfrak g},{\mathfrak h})\to \underline{\Hom}^{p,q+1}({\mathfrak g},{\mathfrak h})
%    \]
%    given by
%    \[
%        (d_{0,1}^{}f)(\gamma_1\wedge\cdots\wedge\gamma_{q+1})=
%        \sum_{i<j}\pm f([\gamma_i,\gamma_j]^{}_{\mathfrak g}\wedge
%        \gamma_1\wedge\cdots
%        \wedge\widehat{\gamma_i}\wedge\cdots\wedge\widehat{\gamma_j}\wedge
%        \cdots\wedge\gamma_{q+1}^{}).
%    \]    
%    }    
    %
\begin{remark}[Sign conventions]
    Notice that our convention slightly differs from \cite{Fiorenza2009} and the reference therein in several aspects:
    \begin{itemize}
        \item In our definition of $\underline{\Hom}^{a,d}(V,W)$ in \eqref{eq:biggraded-Hom-arity-degree} we make an opposite choice for the order of the indices, taking the arity as the first (horizontal) index and the degree as the second (vertical) one.
        \item In equation \eqref{eq:bracket-Hom-DGLAs} we have an explicit sign prefactor ${(-1)^{d_1 (a_2 + d_2) }}$ that is not present in \cite{Fiorenza2009} which ensure in particuar the graded skew-symmetry of $[\blank,\blank]_{\underline{\Hom}}$ with respect to the total degree.
        \item In equation \eqref{eq:vertical-differential-Hom-DGLAs} we have an explicit sign prefactor ${(-1)^{a+d-1}}$ which, together with the sign implied by the operator $\ca$ appearing in equations \eqref{eq:vertical-differential-Hom-DGLAs} and \eqref{eq:horizontal-differential-Hom-DGLAs} (see equation \eqref{eq:RN-product}), ensure that $\d_{\textrm{tot}}=\d_{(1,0)} + \d_{(0,1)}$ yields an honest differential on the total complex.
        Moreover, this sign choice guarantees that when $f$ is a component of an $L_\infty$-morphism from $\g$ to $\h$, then ${a + d -1 =0}$.
        \item All together, these sign conventions guarantee that the Maurer-Cartan equation in the total DGLA $\CE(\g,\h)$ exactly reproduces the defining relations for $L_\infty$-morphisms between DGLAs as in \eqref{eq:Linfty-morphism-DGLAs}.
    \end{itemize}
\end{remark}

 By taking the total grading, one obtains a differential graded Lie algebra structure on the total complex.
\begin{definition}[Chevalley-Eilenberg-type DGLA associated with a pair of DGLAs]\label{def:CE-DGLA-Linfty-morphisms}
    The Chevalley-Eilenberg-type DGLA $\CE(\g,\h)$ associated with the pair of DGLAs $(\g,\h)$  
     is the total DGLA $\big( \underline{\Hom}(\g,\h),\d_{\mathfrak{ce}},[\cdot,\cdot]_{\mathfrak{ce}}\big)$ of the bigraded DGLA 
    $\big( \underline{\Hom}(\g,\h),\d_{(1,0)},\d_{(0,1)},[\cdot,\cdot]_{\underline{\Hom}}\big)$.
\end{definition}
\begin{remark}
    Notice that an homogeneous multilinear map $\underline{\Hom}^{a,d}(\g,\g)$ has total degree $(d + a)$  when seen as an element in $\CE(\g,\h)$ and degree $(d+a-1)$ when seen as element of the Nijenhuis-Richardson algebra of \autoref{def:NR-algebra}. In other words, the graded vector space underlying $\CE(\g,\g)$ is given by the suspension of the graded vector space underlying to $M^{\textrm{skew}}(\g)$.
\end{remark}
\begin{lemma}\label{lem:Linfty-morphisms-as-MC-elements}
    There is a one-to-one correspondence between the set of Maurer-Cartan elements $\MC(\CE(\g,\h))$ and the set $L_\infty(\g,\h)$ of $L_\infty$-morphisms from $\g$ to $\h$.
\end{lemma}
\begin{proof}
    It is clear from \autoref{def:Linfty-morphism-DGLAs} that every component $\Phi_k$ of a given $L_\infty$-morphism $\Phi$ from $\g$ to $\h$ is a homogeneous degree $1$ element in $\CE(\g,\h)$.
    Denote by $\Phi$ the sum of all its components, 
    The Maurer-Cartan equation in $\CE(\g,\h)$ for $\Phi$ reads as
    \begin{displaymath}
        d_{\textrm{tot}} \Phi  + \frac{1}{2} [\Phi,\Phi]_{\underline{\Hom}} = 0~,
    \end{displaymath}
    where $\d_{\textrm{tot}}f = \d_{(1,0)}f +  \d_{(0,1)}f$, for any $f \in \underline{\Hom}^{a,d}(V)$ is the total differential in $\CE(\g,\h)$.
    The latter implies an equation for every value of arity $k\geq 1$ coinciding with \eqref{eq:Linfty-morphism-DGLAs} since the $k$-ary component of    $[\Phi,\Phi]_{\underline{\Hom}}$ corresponds to 
    \begin{align*}
        \sum_{i=1}^{k-1} {{-}(-1)^{(1-i)(k-i + 1 -k +i)}} [\cdot,\cdot]_\h \circ (\Phi_i \otimes \Phi_{k-i}) \circ P_{i,k-i} 
        =&~
        \sum_{i=1}^{k-1} {{-}(-1)^{(1-i)}} [\cdot,\cdot]_\h \circ (\Phi_i \otimes \Phi_{k-i}) \circ P_{i,k-i} 
        \\
        =&~
        {-}
        2 [\cdot,\cdot]_\h \circ \Sop_{2,k}(\Phi)~,
    \end{align*}
    and the $k$-ary component of $\d_{\textrm{tot}}\Phi$ is given by $\Phi_{k} \ca \d_\g + \Phi_{k-1} \ca [\cdot,\cdot]_\g -   \d_\h \circ \Phi_{k}$  since ${k + \deg(\Phi_k) -1 =0}$.
    
\end{proof}

%-.-.-.-.-.-.-.-.-.-.-.-.-.-.-.-.-.-.-.-.-.-.-.-.-.-.-.-.-.-.-+
\subsection{Homotopy equivalent \texorpdfstring{$L_\infty$}{L-infinity}-morphisms between DGLAs}
\label{appendix:MC-homotopy}
%-.-.-.-.-.-.-.-.-.-.-.-.-.-.-.-.-.-.-.-.-.-.-.-.-.-.-.-.-.-.-+
%
%    \antonio{Piccolo cappello categoriale.
%        \\- la DGLA "intervallo" (\cite[Ex. 5.5.5, Ex 5.6.9]{Manetti2022})
%        \\- DGLA come categoria omotopica \cite[Def~3.5.11]{Manetti2022}
%        \\- omotopia definita tramite l'intervallo
%    }
   One has a natural notion of homotopy between $L_\infty$-morphisms. While this is genarlaly true for $L_\infty$-morphisms between arbitrary $L_\infty$-algebras, hre we content ourselves with considering the case of $L_\infty$-morphisms between DGLAs, that is all we need in the present article.
    \begin{definition}[Homotopy between \texorpdfstring{$L_\infty$}{L-infinity}-morphisms between DGLAs]
     Let $\g,\h$ be two DGLAs and let $f,g$ be two $L_\infty$-morphisms from $\g$ to $\h$.
    An \emph{\(L_\infty\)-homotopy} from \(f\) to \(g\) is an \(L_\infty\)-morphism 
    \begin{displaymath}
        H: \g \to \h\otimes \Omega^\bullet(\Delta^1),% \otimes \Omega^\bullet(\Delta^1)~,
    \end{displaymath}
    where $\Omega^\bullet(\Delta^1)\cong \mathbb{k}[t,dt]$ is the graded commutative algebra of (polynomial) differential forms on the 1-simplex,
    such that the following diagram commutes:
    \begin{displaymath}
        \begin{tikzcd}
            &[1em] & [1em] \h 
            \\
            \g \ar[r,"H"] \ar[urr,bend left, "f"] \ar[drr, bend right, "g"'] 
            &  \Omega^\bullet(\Delta^1) \ar[dr, "{\mathrm{ev}_0}"']\ar[ur, "{\mathrm{ev}_1}"] 
            \\
            & & \h 
        \end{tikzcd}
    \end{displaymath}
    where $\mathrm{ev}_{0}, \mathrm{ev}_{1}$  is the evaluation map at the boundary points $0$ and $1$ of the 1-simplex $\Delta^1$.
    We denote with $\Hom_{\infty-{\mathrm{hotop}}}(\g,\h)$ the set of all \(L_\infty\)-homotopies between \(L_\infty\)-morphisms from \(\g\) to \(\h\).
    \end{definition}
    \begin{remark}[$L_\infty$-homotopies as Maurer-Cartan elements]\label{rmk:Linfty-homotopies-as-MC-elements}
        According to \autoref{lem:Linfty-morphisms-as-MC-elements}, one can characterize \(L_\infty\)-homotopies between DGLAs as Maurer-Cartan elements in $\CE(\g,\h[t,\d t])$.
        Notice that $\CE(\g,\h[t,\d t])$ can be identified with the DGLA of polynomial differential forms on the interval with values in $\CE(\g,\h)$, that is, one has a natural identification $\CE(\g,\h[t,\d t])\cong (\CE(\g,\h))[t,\d t]$. %since
%        \begin{align*}
%            \underline{\Hom}^{a,d}(\g,\h[d])
%            &\cong~
%            \Hom_{\gVect}\Big(\wedge^a \g,(\h\otimes \Omega^\bullet(\Delta^1))[d]\Big) 
%            \\
%            &\cong~
%            \Big(\Hom_{\gVect}(\wedge^a \g,\h[d])\Big)\otimes \Omega^\bullet(\Delta^1)
%            \\
%            &\cong~
%            \Big(\Hom_{\gVect}(\wedge^a \g,\h[d])\Big)[t,\d t] 
%        \end{align*}
 %   \end{remark}

    %
%    \antonio{intermezzo categoriale.
%        \\
%        - DGLA come 2-categoria
%        \\
%        - Omotopie come 2-morfismi.
%    }
% Ripetere Gauge equivalence qui. Non c'è più nelle conventions.
%, meaning that two Maurer-Cartan elements $x,y\in\MC(\g)$ are gauge equivalent if and only if there exists $a\in\g_0$ such that $y=e^a\star x$.
%    \begin{remark}
 %Besides the notion of gauge equivalence recalled in the introduction, there is another instance of
 One has a natural notion of homotopy  equivalence between Maurer-Cartan elements for a DGLA $\mathfrak{k}$:
        two Maurer-Cartan elements $\alpha_0,\alpha_1 \in \MC(\mathfrak{k})$ are said to be homotopy equivalent if there exists a Maurer–Cartan element 
        $$
        \alpha(t,dt)\in \mathfrak{k}\otimes \Omega^\bullet(\Delta^1)
        $$
        such that $\alpha(0)=\alpha_0$ and $\alpha(1)=\alpha_1$.
       One tautologically sees that that two $L_\infty$-morphisms $f,g:\g \to \h$ between the DGLAs $\g$ and $\h$ are homotopy equivalent as $L_\infty$-morphisms if and only if $f$ and $g$ are  homotopy equivalent as Maurer-Cartan elements in $\MC(\CE(\g,\h))$.
    \end{remark}      

%    \domenico{From \href{https://mathoverflow.net/questions/48552/homotopy-between-solutions-of-maurer-cartan-equation}{mathoverflow}:
%        \\
%        At first sight, it is not apparent that this defines an equivalence relation, but indeed it does.  
%        In fact, much more is true: the homotopy relation just described is only the tip of the iceberg.  
%        To see this, one rewrites $\g[t,dt]$ as $\g\otimes\Omega^1$, where $\Omega^1$ denotes the differential graded commutative algebra of polynomial differential forms on the (algebraic) $1$-simplex.  
%        This is naturally the first stage of a simplicial DGLA 
%        \[  \g\otimes\Omega^\bullet, \]
%        and taking Maurer–Cartan elements in it produces a simplicial set 
%        \[ \MC(\g\otimes\Omega^\bullet). \]
%        This simplicial set turns out to be a Kan complex, and the fact that the homotopy relation between solutions of the Maurer–Cartan equation in $\g$ is an equivalence relation corresponds precisely to the \emph{horn-filling property} of this Kan complex.  
%
%        A good reference for these ideas is Getzler’s paper \cite{Getzler2009}.
%        In formal deformation theory, one often obtains a nilpotent DGLA from a general one by tensoring it with the maximal ideal $m_A$ of a local Artin algebra $A$.
%}
%
A crucial result in infinitesimal deformation theory via DGLAs and $L_\infty$-algebras is the fact that gauge equivalence and homotopy equivalence of Maurer-Cartan elements in a DGLA coincide  (see, e.g., \cite{Manetti2022}). In particular, it is immediate to see that if the Maurer-Cartan elements $f,g\in MC(\mathfrak{k})$ are gauge equivalent via the gauge equivalence $e^x$, i.e., if $e^x\star f=g$, then the Maurer-Cartan element $e^{tx}\star f\in MC(\mathfrak{k}\otimes \Omega^\bullet(\Delta^1))$ is a homotopy between $f$ and $g$.
Summing up, one gets the following.
\begin{proposition}[{\cite[\S 3]{Fiorenza2009}}]
    \label{prop:Linfty-homotopy-gauge-equivalence}
    Let $f,g:{\g} \to {\h}$ be two $L_\infty$-morphisms of DGLAs. Then $f$ and $g$ are homotopic $L_\infty$-morphisms if and only if 
    $f$ and $g$ are gauge equivalent as elements in $\MC(\CE(\g,\h))$. 
    \end{proposition}
%begin{proof}
 %   
%
%    \antonio{chiedere a Domenico se e come scrivere la dimostrazione}
%    \domenico{
%    Dovrebbe essere facile provare che se $\exists H \in \MC(\g[t,dt])$ con $H|_{t=0}=f$ e $H|_{t=1}=g$ allora $\exists x \in \g_0$ tale che $e^x \star f =g$.
%
%    definisco $x \mapsto H:= e^{t x} \star f$. Poichè $f \in \MC(\g)\subseteq \MC(\g[t,dt])$ allora $tx \in \Big( \g[t,dt] \Big)^0$. dunque $e^{t x} \star f \in \MC(\g[t,dt])$.
%
%    Viceversa più spinoso, serve $x = \int H$...
%    }
%\end{proof}
%
Expanding the definition of gauge equivalence, the above Proposition states that two $L_\infty$-morphisms $f,g:\g\to \h$ are homotopic  if and only if there exists an element $\mathcal{h} \in \CE(\g,\h)^0$ such that %$\Phi$ and $\Psi$ are related by the gauge transformation generated by $\mathcal{h}$, (see equation \eqref{eq:MC-Gauge-Transformation}), i.e.
\begin{align}
        g=&%~ 
        %e^{{\mathcal{h}}} \star \Psi 
        %\notag
        %\\
        %=&
        ~
         f+ \sum_{n\geq 0}\frac{[\mathcal{h}, \blank]_{\mathfrak{ce}}^{n}}{(n+1)!}([\mathcal{h}, f]_{\mathfrak{ce}} - \d_{\mathfrak{ce}}(\mathcal{h}))
        ~.
        \label{eq:MC-Gauge-Transformation-CE}
\end{align}
%    Unpacking the definitions, one finds that $\mathcal{h}$ consists of a collection of multilinear maps, for $k\geq 1$,
%    \begin{displaymath}
%        \mathcal{h}_k : \wedge^k \g \to \h[-k]~,
%    \end{displaymath}
%    satisfying the following compatibility conditions for every $k\geq 1$:
%    \begin{displaymath}
%        \Phi_k - \Psi_k =~
%        \d_\h \circ \mathcal{h}_k
%        +
%        (-1)^{k} \mathcal{h}_k \circ \d_\g
%        +
%        \sum_{i+j=k}
%        (-1)^{i j} [\cdot,\cdot]_\h \circ (\Psi_i \otimes \mathcal{h}_j) \circ P_{i,j}
%        +
%        \sum_{i+j=k}
%        (-1)^{i j} [\cdot,\cdot]_\h \circ (\mathcal{h}_i \otimes \Phi_j) \circ P_{i,j}~.
%    \end{displaymath}
%    \antonio{Da ricontrollare!!!}
%
 \subsection{Homotopy fibers of \texorpdfstring{$L_\infty$}{L-infinity}-morphisms between DGLAs}
   % \antonio{Secondo me serve dire qualcosa sulle fibre omotopiche per rendere il paper più accessibile!}
Defining the notion of a homotopy between two $L_\infty$-morphisms is the first step in defining a whole $\infty$-category of $L_\infty$-algebras, see, e.g., \cite{LurieFormalModuli,Pridham2010}. In particular, given an $L_\infty$-morphism $f:\g \to \h$ between DGLAs, it makes sense to consider its homotopy fiber $\hoFib(f)$, i.e, the homotopy pull-back 
    \begin{displaymath}
        \begin{tikzcd}[column sep=large]
            \hoFib(f) \ar[r] \ar[d] \ar[dr,very near start,phantom,"\lrcorner_{\mathcal{h}}"] & 0 \ar[d] \\
            \g \ar[r, "f"'] & \h %\ar[ur,"\mathcal{h}", Rightarrow] & \h
        \end{tikzcd}
    \end{displaymath}
    where $0$ is the trivial DGLA.
    By definition,  this means that one has an $L_\infty$-morphism $\varphi$ and an $L_\infty$-homotopy $\mathcal{h}$ such that $\hoFib(f)$ fits into a homotopy commutative diagram of the form
    \begin{displaymath} 
        \begin{tikzcd}[column sep=large]
            \hoFib(f) \ar[r] \ar[d,"\varphi"'] & 0 \ar[d] \\
            \g \ar[r, "f"'] \ar[ur, Rightarrow, "\mathcal{h}"] & \h
        \end{tikzcd}
    \end{displaymath}
 and  $\hoFib(f)$ is universal with respect to this property:
 %   \\
   % Universality means that for any other $L_\infty$-algebra $\mathfrak{k}$ fitting into a similar 
   for any homotopy commutative diagram of $L_\infty$-algebras of the form
    \begin{displaymath} 
        \begin{tikzcd}[column sep=large]
            \mathfrak{k} \ar[r] \ar[d,"\varphi_{\mathfrak{k}}"'] & 0 \ar[d] \\
            \g \ar[r, "f"'] \ar[ur, Rightarrow, "\mathcal{h}_{\mathfrak{k}}"'] & \h
        \end{tikzcd}
    \end{displaymath}
    one has a unique (up to homotopy) factorization %$L_\infty$-morphism $\Xi:\mathfrak{k} \to \hoFib(\Phi)$ making the following diagram 
    into a homotopy commutative diagram of $L_\infty$-algebras of the form
    \begin{displaymath} 
        \begin{tikzcd}[column sep=large]
            \mathfrak{k}  \ar[dr,"\Xi"] \ar[ddr,bend right=30, "\varphi_{\mathfrak{k}}"',""{name=D}] \ar[rrd,bend left=30,""{name=U,below}] & &
            \\
            & \hoFib(f) \ar[r] \ar[d,"\varphi"'] \arrow[Rightarrow, to=U, ""] \ar[Rightarrow, from=D]
            & 0 \ar[d]
            \\
            & \g \ar[r, "f"'] \ar[ur, Rightarrow, "\mathcal{h}"] & \h
        \end{tikzcd}
    \end{displaymath}
Notice that, being characterized by a universal property in a $\infty$-category, the $L_\infty$-algebra $\hoFib(f)$ is unique up to homotopy equivalence. A representative for the homotopy class of $\hoFib(f)$ will be called a \emph{model} for the homotopy fiber of $f$.

%    \antonio{Todo:
%    \\- la questione dei modelli e/o  esistenza e unicità a meno di omotopia
%    }
%
%    Being the above homotopy fibers a particular instance of hotopy limits (pull-backs), they are unique up to homotopy equivalence, i.e., if $\hoFib(\Phi)$ and $\hoFib'(\Phi)$ are two homotopy fibers of the same $L_\infty$-morphism $\Phi:\g \to \h$, then there exists an $L_\infty$-quasi-isomorphism between them.
%    \todo{Check!}
%    A representative of such a homotopy class is often referred to as \emph{a model} for the homotopy fiber of $\Phi$.
% 
%    \antonio{riferimenti finali al libro di Manetti}
%

%####################################################################%
% BIBLIOGRAPHY
%####################################################################%
\printbibliography
%####################################################################%

\end{document}